\documentclass[a4paper,10pt,leqno,oneside]{amsart}
\usepackage{amsmath}
\usepackage{amsfonts}
\usepackage{enumerate}
\usepackage{amsthm}
\usepackage{bm}
\usepackage{mathtools}
\usepackage{amssymb}
\usepackage[top=20truemm, bottom=20truemm, left=15truemm, right=15truemm]{geometry}

\usepackage[usenames,dvipsnames]{color}
\usepackage{xcolor}
\makeatletter \@addtoreset{equation}{section} \makeatother
\newcommand{\eref}[1]{(\ref{#1})}
\newcommand{\tref}[1]{Theorem \ref{#1}}
\newcommand{\pref}[1]{Proposition \ref{#1}}
\newcommand{\lref}[1]{Lemma \ref{#1}}
\newcommand{\cref}[1]{Corollary \ref{#1}}
\newcommand{\rref}[1]{Remark \ref{#1}}

\theoremstyle{plain} \newtheorem{thm}{Theorem}[section] \newtheorem{lem}{Lemma}[section] \newtheorem{prop}{Proposition}[section] \newtheorem{cor}{Corollary}[section]
\theoremstyle{definition} \newtheorem{rem}{Remark}[section] 
\colorlet{RED}{red}
\title[Regularity of Bingham problem]{$H^2$-regularity for stationary and non-stationary Bingham problems\\with perfect slip boundary condition}
\author[Takeshi Fukao]{Takeshi Fukao}
\address{Faculty of Advanced Science and Technology, Ryukoku University, 1-5 Yokotani, Seta Oe-cho, Otsu-shi, 520-2194 Shiga, Japan}
\email{fukao@math.ryukoku.ac.jp}
\author[Takahito Kashiwabara]{Takahito Kashiwabara}
\address{Graduate School of Mathematical Sciences, The University of Tokyo, 3-8-1 Komaba, Meguro, 153-8914 Tokyo, Japan}
\email{tkashiwa@ms.u-tokyo.ac.jp}
\makeatletter
\@namedef{subjclassname@2020}{\textup{2020} Mathematics Subject Classification}
\makeatother

\begin{document}
\begin{abstract}
	$H^2$-spatial regularity of stationary and non-stationary problems for Bingham fluids formulated with the pseudo-stress tensor is discussed.
	The problem is mathematically described by an elliptic or parabolic variational inequality of the second kind, to which weak solvability in the Sobolev space $H^1$ is well known.
	However, higher regularity up to the boundary in a bounded smooth domain seems to remain open.
	This paper indeed shows such $H^2$-regularity if the problems are supplemented with the so-called perfect slip boundary condition and if the yield stress vanishes on the boundary.
	For the stationary Bingham--Stokes problem, the key of the proof lies in a priori estimates for a regularized problem avoiding investigation of higher pressure regularity, which seems difficult to get in the presence of a singular diffusion term.
	The $H^2$-regularity for the stationary case is then directly applied to establish strong solvability of the non-stationary Bingham--Navier--Stokes problem, based on discretization in time and on the truncation of the nonlinear convection term.
\end{abstract}
\maketitle

\section{Introduction}
Bingham fluid, one of non-Newtonian fluids, is a model describing the motion of viscoplastic materials, which behave like viscous fluid if shear stress magnitude exceeds a threshold, called the yield stress, and like rigid bodies otherwise.
Examples of viscoplastic materials are sauces, pastes, crude oils, drilling muds, slurries, lava, etc., understanding of which is important in several fields of industry or physics.

Let $\Omega \subset \mathbb R^d \, (d = 2, 3)$ be a bounded domain with the smooth boundary $\Gamma := \partial\Omega$.
We consider the following boundary value problem:
\begin{align}
	&- \operatorname{div} (\bm T^D - p \bm I) = \bm f, \quad \operatorname{div} \bm u = 0 \quad\text{in}\quad \Omega, \label{eq: strong form} \\
	&\quad \begin{cases}
		\bm T^D = \nu \nabla \bm u + g \dfrac{\nabla \bm u}{|\nabla \bm u|} &\text{if}\quad \nabla \bm u \neq \bm O, \\
		|\bm T^D| \le g &\text{if}\quad \nabla \bm u = \bm O,
	\end{cases} \label{eq: TD} \\
	&\bm u \cdot \bm n = 0, \quad (\bm T^D \bm n)_\tau = \bm 0 \quad\text{on}\quad \Gamma, \label{eq: slipBC}
\end{align}
which we refer to as the \emph{stationary Bingham--Stokes} problem (note that it reduces to the standard Stokes equations if $g = 0$), and the following initial boundary value problem:
{\allowdisplaybreaks \begin{align}
	& \frac{\partial \bm u}{\partial t}+ (\bm u \cdot \nabla ) \bm u- \operatorname{div} (\bm T^D - p \bm I) = \bm f, \quad \operatorname{div} \bm u = 0 \quad\text{in}\quad Q:=(0,T) \times \Omega, \label{eeq: strong form} \\
	&\quad \begin{cases}
		\bm T^D = \nu \nabla \bm u + g \dfrac{\nabla \bm u}{|\nabla \bm u|} &\text{if}\quad \nabla \bm u \neq \bm O, \\
		|\bm T^D| \le g &\text{if}\quad \nabla \bm u = \bm O,
	\end{cases} \label{eeq: TD} \\
	&\bm u \cdot \bm n = 0, \quad (\bm T^D \bm n)_\tau = \bm 0 \quad\text{on}\quad \Sigma:=(0,T) \times \Gamma, \label{eeq: slipBC} \\
	& \bm u(0) = \bm u_0 \quad \text{in} \quad \Omega, \label{eeq: IC}
\end{align}}
which we refer to as the \emph{non-stationary Bingham--Navier--Stokes} problem.
Here, the unknown functions $\bm T^D$, $\bm u$, and $p$ mean the deviatoric part of the stress tensor ($\nu > 0$ is the viscosity constant), the velocity, and the pressure of a material, respectively.
Then $\bm T = -p \bm I + \bm T^D$ constitutes the total stress tensor with $\bm I$ the identity matrix.
For matrices $\bm A, \bm B \in \mathbb R^{d \times d}$, we define the inner product and magnitude by $\bm A : \bm B = \sum_{i, j = 1}^d A_{ij} B_{ij}$ and $|\bm A| = (\bm A : \bm A)^{1/2}$, respectively.
The external force $\bm f$ and the yield stress $g \ge 0$ are given functions.

Let us discuss the stationary problem for a while.
Compared with the standard Bingham fluid model (see e.g.  \cite[Section VI.1.3]{DuLi1976}), problem \eref{eq: strong form}--\eref{eq: TD} is simplified in the sense that we exploit the \emph{pseudo stress tensor}, i.e, we use $\nabla \bm u = (\partial u_i/\partial x_j)_{1\le i, j \le d}$ instead of the rate-of-strain tensor $\bm D(\bm u) = (\nabla \bm u + (\nabla \bm u)^\top)/2$ in the constitutive law.
This facilitates technical arguments in our analysis (especially, the coercivity relation is simpler than the one obtained from Korn's inequality, cf.\ \eref{eq: coercivity} below), still retaining mathematical difficulties arising from the Bingham nonlinearity.

The boundary condition \eref{eq: slipBC} is the perfect slip boundary condition, where $\bm n$ is the outer unit normal to $\Gamma$ and the subscript $\tau$ means the tangential part of a vector (namely, $\bm v_\tau = \bm v - (\bm v \cdot \bm n) \bm n$).
In view of \eref{eq: TD}, it reduces to
\begin{equation} \label{eq2: slipBC}
	\bm u \cdot \bm n = 0, \quad (\nabla \bm u \, \bm n)_\tau = \bm 0 \quad\text{on}\quad \Gamma.
\end{equation}

By adopting the theory presented in the classical literature \cite[Chapter VI]{DuLi1976}, we can easily find that \eref{eq: strong form}--\eref{eq: slipBC} is formulated as a variational inequality problem and that it admits a unique weak solution $\bm u \in \bm H^1(\Omega)$.
However, the higher regularity, e.g., $\bm u \in \bm H^2(\Omega)$ seems much less understood.
For the scalar Bingham problem without the divergence-free constraint (also known as Bingham flows in cylinders), the $H^2$-regularity is established by \cite[Theorem 15]{Bre1971} using a sophisticated estimate for a resolvent-like problem for the Laplace operator.
We also refer to recent results \cite{AMS2023, AMS2024}, which show $H^2(\Omega)$-regularity for a scalar pseudo-parabolic problem with a singular diffusion term (similar to the Bingham problem).
The strategy of \cite{AMS2023} has several things in common with ours; e.g., regularity for a stationary problem is obtained first, which is then incorporated to examine an evolution problem via discretization in time.
From another perspective, partial interior $C^1$-regularity, as well as interior $H^2$-regularity (i.e.\ $\bm u \in \bm H^2_{\mathrm{loc}}(\Omega)$), is obtained in \cite{FuSe1998, FuSe2000} for the stationary Bingham flow with the convective term.
See also \cite{Kim1989} which mentions some interior $H^2$-regularity for the non-stationary problem.
More recently, the $C^1(\Omega)$-regularity is established for the scalar case by \cite{Tsu2024}.

On the other hand, it seems that there are no results of $H^2$-regularity \emph{up to the boundary} when the divergence-free constraint is imposed, that is,  the pressure $p$ appears as an unknown variable.
A major difficulty consists in obtaining higher regularity than $L^2(\Omega)$ for $p$.
In fact, as is seen from the linear Stokes case, we usually need to establish higher tangential regularity of $p$ before discussing the regularity of $\bm u$ in the normal direction (see \cite[pp.\ 24--25]{CoFo1989} and \cite[p.\ 1104]{BdV04}).
However, this strategy seems infeasible for the Bingham case, which formally reads 
$\nabla p = \bm f - \nu \Delta \bm u - \operatorname{div}(g \nabla \bm u/|\nabla \bm u| )$, since the last term is too singular to be even a measurable function near the region where $\nabla\bm u$ vanishes.

If the no-slip boundary condition is imposed, it is not trivial to settle this issue.
Thereby we suggest to replace it by the perfect slip one \eref{eq: slipBC} or \eref{eq2: slipBC} and indeed prove the regularity $\bm u \in \bm H^2(\Omega)$ up to the boundary if $g \in H^1_0(\Omega)$, which is the first objective of the present paper (see \cref{cor: main result} below).
In the proof, we introduce a regularized problem, transform it into a half-space problem, and estimate tangential and normal second-order derivatives.
The key observation is that for \eref{eq2: slipBC} one can obtain the regularity of velocity in the normal direction, knowing only lower regularity for the pressure (in other words, higher regularity is shown only for $\bm u$ but not for $p$).

Let us make a comment on the assumption $g = 0$ on $\Gamma$, which is required for a technical reason (see \rref{rem: H^1_0 is inevitable}(ii)).
We emphasize that, even with it the issue of $\bm H^2$-regularity up to the boundary remains non-trivial.
In fact, once we are away from the boundary and $g > 0$---no matter how small the distance or $g$ is---the effect of singular diffusion becomes present, which cannot be reduced to an interior case or to small perturbation from the standard Stokes equations.
However, it is true that this assumption excludes some important case, e.g.\ $g = $ const., which is left as an open problem.

The latter part of this paper (Section \ref{sec: non-stationary}) deals with the non-stationary Bingham--Navier--Stokes problem.
In this regard we notice that preceding studies \cite{BL14, DuLi1976, Fukao2013, Kato1993, Kim1987, MMD09, NW79} considered only one spatial differentiability whereas $H^2$-regularity in the two-dimensional whole space was obtained in \cite{Kim1986}.
By directly applying \cref{cor: main result}, we will prove that there exists a unique local-in-time solution $\bm u$ having the $H^2$-regularity---more precisely, $\bm u \in W^{1, \infty}(0, T; \bm L^2(\Omega)) \cap L^\infty(0, T; \bm H^2(\Omega))$ in \tref{thm: evolution} and $\bm u \in H^1(0, T; \bm L^2(\Omega)) \cap L^2(0, T; \bm H^2(\Omega))$ in \tref{thm: evolution weak}---depending on assumptions for the data $\bm f, g$, and $\bm u_0$.
The proof is based on the discretization in time (also known as Rhothe's method) combined with the truncation technique for the convection term as in \cite{Bar10}.

Briefly speaking, our well-posedness result tells us that, if the yield stress $g$ is first-order differentiable with respect to space and time, then we can construct an $H^2$-strong solution.
As one of its possible applications and future studies, we mention non-isothermal Bingham fluids involving heat transfer, in which the yield stress and the viscosity parameter may depend on temperature (cf.\ \cite{Gonz2020, Mes14}).

We use the standard Lebesgue spaces $L^p(\Omega)$ and Sobolev spaces $W^{m, p}(\Omega)$ for $p \in [1, \infty]$ and $m \in \mathbb N$.
Those spaces on the boundary, in particular the Sobolev--Slobodeckij space $W^{s, p}(\Gamma)$ will also be utilized.
Function spaces for vectors are indicated with bold fonts, e.g., $\bm L^p(\Omega), \, \bm W^{m, p}(\Omega)$.
Those for tensors (matrices) are also indicated by the same symbols; nevertheless, in case confusion may occur, we write $L^p(\Omega)^{d \times d}$ etc.
Throughout this paper, we let $C$ denote a generic positive constant that may differ at each occurrence.
When we need to specify its dependency on parameters, we indicate $C = C(\Omega)$ etc.

\section{Weak formulations for the stationary problem}
In this and next sections, let us consider the stationary Bingham--Stokes problem \eref{eq: strong form}--\eref{eq: slipBC}.
For simplicity in the presentation, we assume $\nu = 1$.
Let us derive weak formulations and recall unique existence of a weak solution, which is already known.
For that purpose we introduce the following Hilbert spaces:
\begin{align*}
	\bm V &= \{ \bm v \in \bm H^1(\Omega) \mid \bm v \cdot \bm n = 0 \; \text{ on } \Gamma \}, \\
	\bm V_\sigma &= \{ \bm v \in \bm V \mid \operatorname{div} \bm v = 0 \}, \quad \|\bm v\|_{\bm V} = \|\bm v\|_{\bm V_\sigma} := \|\bm v\|_{\bm H^1(\Omega)}, \\
	\bm H_\sigma &= \{ \bm v \in \bm L^2(\Omega) \mid \operatorname{div} \bm v = 0, \quad \bm v \cdot \bm n = 0 \; \text{ on } \Gamma \}, \quad \|\bm v\|_{\bm H_\sigma} := \|\bm v\|_{\bm L^2(\Omega)}, \\
	L^2_0(\Omega) &= \Bigl\{ q \in L^2(\Omega) \mid \int_\Omega q \, d\bm x = 0 \Bigr\}.
\end{align*}
We regard $\bm V_\sigma \subset \bm H_\sigma = \bm H_\sigma' \subset \bm V_\sigma'$ as a Gelfand triple, where the prime means the dual norm space.
The $L^2(\Omega)$-inner products for scalar-, vector-, and matrix-valued functions are denoted by the same symbol $(\cdot, \cdot)$.
The duality pairing between $\bm V_\sigma'$ and $\bm V_\sigma$ is written as $\left< \cdot, \cdot \right>$.

With this setting, we show that a solution of the Bingham problem satisfies a variational inequality.

\begin{prop} \label{prop: VIsigma}
	(i) Let $(\bm u, p)$ be a sufficiently smooth solution of \eref{eq: strong form}--\eref{eq: slipBC}.
	Then we have
	\begin{equation} \label{eq: VI}
		(\nabla \bm u, \nabla (\bm v - \bm u)) + (g, |\nabla \bm v| - |\nabla \bm u|) \ge \left< \bm f, \bm v - \bm u \right> \qquad \forall \bm v \in \bm V_\sigma.
	\end{equation}
	
	(ii) Let $\bm f \in \bm V_\sigma'$ and $g \in L^2(\Omega), \, g \ge 0$.
	Then there exists a unique solution $\bm u \in \bm V_\sigma$ of \eref{eq: VI}.
\end{prop}
\begin{proof}
	(i) It follows from integration by parts that $(\bm T^D(\bm u), \nabla \bm v) = (\bm f, \bm v)$ for all $\bm v \in \bm V_\sigma$, and hence
	\begin{align*}
		(\nabla \bm u, \nabla (\bm v - \bm u)) + (g, |\nabla \bm v| - |\nabla \bm u|) - (\bm f, \bm v - \bm u)
		= (g, |\nabla \bm v| - |\nabla \bm u|) - (\bm T^D(\bm u) - \nabla \bm u, \nabla (\bm v - \bm u)).
	\end{align*}
	The right-hand side $\ge 0$ because $(\bm T^D(\bm u) - \nabla \bm u) : \nabla \bm v \le g|\nabla \bm v|$ and $(\bm T^D(\bm u) - \nabla \bm u) : \nabla \bm u = g|\nabla \bm u|$ in $\Omega$ by \eref{eq: TD}.
	
	(ii) Observe that the functional $j$ defined by $j(\bm v) = (g, |\nabla \bm v|)$ is continuous from $\bm V$ to $\mathbb R$ and convex, and that the following coercivity holds:
	\begin{equation} \label{eq: coercivity}
		\|\nabla \bm v\|_{\bm L^2(\Omega)}^2 \ge C \|\bm v\|_{\bm H^1(\Omega)}^2 \qquad \forall \bm v \in \bm V.
	\end{equation}
	Then a well-known existence theory of elliptic variational inequalities (see \cite[Theorem I.4.1]{Glo2008}) implies the desired result.
\end{proof}
\begin{rem}
	There holds a Poincar\'e inequality $\|\bm v\|_{\bm L^2(\Omega)} \le C\|\nabla \bm v\|_{\bm L^2(\Omega)}$ for $\bm v \in \bm H^1(\Omega)$ such that $\bm v \cdot \bm n = 0$ on $\Gamma$, which may be proved in a similar way to \cite[Lemma 2.3]{BdV04}.
	In fact, a constant vector field $\bm v$ satisfying $\bm v \cdot \bm n = 0$ on $\Gamma$ must be zero since the Gauss map $\bm n$ is surjective from a compact manifold $\Gamma$ onto the $(d-1)$-dimensional sphere.
\end{rem}

The converse statement of \pref{prop: VIsigma} is not trivial to prove.
We achieve it by constructing a multiplier based on a regularized variational inequality introduced for $\epsilon > 0$ as follows:
\begin{equation} \label{eq: VI eps}
	(\nabla \bm u_\epsilon, \nabla (\bm v - \bm u_\epsilon)) + (g, \sqrt{|\nabla \bm v|^2 + \epsilon^2} - \sqrt{|\nabla \bm u_\epsilon|^2 + \epsilon^2}) \ge \left< \bm f, \bm v - \bm u_\epsilon \right> \quad \forall \bm v \in \bm V_\sigma.
\end{equation}
\begin{prop} \label{prop: VI eps,sigma}
	(i) Let $\bm f \in \bm V_\sigma'$ and $g \in L^2(\Omega), \, g \ge 0$.
	Then there exists a unique solution $\bm u_\epsilon \in \bm V_\sigma$ of \eref{eq: VI eps}.
	Moreover, it satisfies the following variational equation:
	\begin{equation} \label{eq: VE eps,sigma}
		(\nabla \bm u_\epsilon, \nabla \bm v) + \Big( g \frac{\nabla \bm u_\epsilon}{ \sqrt{|\nabla \bm u_\epsilon|^2 + \epsilon^2} }, \nabla \bm v \Big) = \left< \bm f, \bm v \right> \qquad \forall \bm v \in \bm V_\sigma.
	\end{equation}
	
	(ii) We have $\|\bm u_\epsilon\|_{\bm H^1(\Omega)} \le C \|\bm f\|_{\bm V_\sigma'}$.
	Moreover, when $\epsilon \to 0$, $\bm u_\epsilon$ converges to the solution $\bm u$ of \eref{eq: VI} strongly in $\bm V_\sigma$.
\end{prop}
\begin{proof}
	(i) Unique existence of $\bm u_\epsilon$ is immediately obtained as in \pref{prop: VIsigma}(ii).
	Let $j_\epsilon : \bm V \to \mathbb R$ be defined by $j_\epsilon(\bm v) = (g, \sqrt{|\nabla \bm v|^2 + \epsilon^2})$.
	It is G\^ateaux differentiable and we get
	\begin{equation*}
		j_\epsilon'(\bm v)[\bm w] = \lim_{s \to 0} \frac{j_\epsilon(\bm v + s \bm w) - j_\epsilon(\bm v)}{s} = \int_\Omega g \frac{\nabla \bm v}{ \sqrt{|\nabla \bm v|^2 + \epsilon^2} } : \nabla \bm w \, d\bm x \quad (\bm v, \bm w \in \bm V).
	\end{equation*}
	Now take $\bm u_\epsilon + s \bm v$ with $s \gtrless 0$ as a test function in \eref{eq: VI eps} to find
	\begin{equation*}
		(\nabla \bm u_\epsilon, \nabla \bm v) + \frac{j_\epsilon(\bm u_\epsilon + s \bm v) - j_\epsilon(\bm u_\epsilon)}{s} \gtreqless \left< \bm f, \bm v \right>.
	\end{equation*}
	Making $s \to \pm0$ deduces the desired equality.
	
	(ii) The $\bm H^1(\Omega)$-estimate for $\bm u_\epsilon$ immediately results from taking $\bm v = \bm u_\epsilon$ in \eref{eq: VE eps,sigma} together with the coercivity \eref{eq: coercivity}.
	Next, taking $\bm v = \bm u_\epsilon$ in \eref{eq: VI} and $\bm v = \bm u$ in \eref{eq: VI eps} respectively and adding the resulting inequalities, we obtain
	\begin{equation*}
		\|\nabla(\bm u - \bm u_\epsilon)\|_{\bm L^2(\Omega)}^2 \le j_\epsilon(\bm u) - j(\bm u) + j(\bm u_\epsilon) - j_\epsilon(\bm u_\epsilon) \le \epsilon \|g\|_{L^1(\Omega)},
	\end{equation*}
	where we have used $0\le \sqrt{|\bm A|^2 + \epsilon^2} - |\bm A| = \epsilon^2/(\sqrt{|\bm A|^2 + \epsilon^2} + |\bm A|) \le \epsilon$.
	This shows $\bm u_\epsilon \to \bm u$ strongly in $\bm V_\sigma$ as $\epsilon \to 0$.
\end{proof}

For the regularized problem, it is easy to recover the associated pressure.
\begin{prop} \label{prop: VEeps}
	If $\bm f \in \bm L^2(\Omega)$, then for the solution $\bm u_\epsilon \in \bm V_\sigma$ constructed in \pref{prop: VI eps,sigma}, there exists a unique $p_\epsilon \in L^2_0(\Omega)$ such that
	\begin{equation} \label{eq: VEeps}
		(\nabla \bm u_\epsilon, \nabla \bm v) + \Big( g \frac{\nabla \bm u_\epsilon}{ \sqrt{|\nabla \bm u_\epsilon|^2 + \epsilon^2} }, \nabla \bm v \Big) - (p_\epsilon, \operatorname{div} \bm v) = (\bm f, \bm v) \qquad \forall \bm v \in \bm V.
	\end{equation}
	Moreover, $p_\epsilon$ satisfies
	\begin{equation} \label{eq: peps is bounded}
		\|p_\epsilon\|_{L^2(\Omega)} \le C (\|\bm f\|_{\bm L^2(\Omega)} + \|g\|_{L^2(\Omega)}).
	\end{equation}
\end{prop}
\begin{proof}
	Unique existence of $p_\epsilon$ follows from an abstract theory for variational equations in a mixed form (see \cite[Lemma I.4.1(ii)]{GiRa1986}), if we notice that the inf-sup condition
	\begin{equation*}
		C \|q\|_{L^2(\Omega)} \le \sup_{\bm v \in \bm H^1_0(\Omega)} \frac{(q, \operatorname{div} \bm v)}{\|\bm v\|_{\bm H^1(\Omega)}} \quad \forall q \in L^2_0(\Omega)
	\end{equation*}
	holds (see e.g.  \cite[p.\ 81]{GiRa1986}).
	Since $\bm H^1_0(\Omega) \subset \bm V$, we also obtain
	\begin{align*}
		C \|p_\epsilon\|_{L^2(\Omega)}
		&\le \sup_{\bm v \in \bm V} \frac{(p_\epsilon, \operatorname{div} \bm v)}{\|\bm v\|_{\bm H^1(\Omega)}}
			= \sup_{\bm v \in \bm V} \frac{ (\nabla \bm u_\epsilon, \nabla \bm v) + \Big( g \frac{\nabla \bm u_\epsilon}{ \sqrt{|\nabla \bm u_\epsilon|^2 + \epsilon^2} }, \nabla \bm v \Big) - (\bm f, \bm v) }{\|\bm v\|_{\bm H^1(\Omega)}} \\
		&\le \|\nabla \bm u_\epsilon\|_{\bm L^2(\Omega)} + \|g\|_{L^2(\Omega)} + \|\bm f\|_{\bm L^2(\Omega)},
	\end{align*}
	which combined with \pref{prop: VI eps,sigma}(ii) proves the $L^2(\Omega)$-estimate for the pressure.
\end{proof}

With the help of the regularized problem, we are now able to establish a counterpart to the ``converse'' statement of \pref{prop: VIsigma} (cf.\ \cite[Theorem II.6.3]{Glo2008}).
\begin{prop}
	Let $\bm f \in \bm L^2(\Omega)$, $g \in L^2(\Omega), \, g \ge 0$, and $\bm u \in \bm V_\sigma$ be the solution of \eref{eq: VI}.
	Then there exist $\bm\lambda \in \bm L^\infty(\Omega) = L^\infty(\Omega)^{d \times d}$ and $p \in L^2_0(\Omega)$ such that
	\begin{align}
		&(\nabla \bm u + g \bm\lambda, \nabla \bm v) - (p, \operatorname{div} \bm v) = (\bm f, \bm v) \qquad \forall \bm v \in \bm V, \label{eq: VE} \\
		&\operatorname{Tr} \bm\lambda = 0, \quad |\bm\lambda| \le 1, \quad g\bm\lambda : \nabla \bm u = g|\nabla \bm u| \quad\text{a.e.\ in}\quad \Omega. \notag
	\end{align}
	In particular, if we set $\bm T^D := \nabla\bm u + g \bm\lambda$ then $(\bm T^D, \bm u, p)$ solves \eref{eq: strong form}--\eref{eq: slipBC}.
\end{prop}
\begin{proof}
	In \eref{eq: VEeps} we set $\bm\lambda_\epsilon := \nabla \bm u_\epsilon/ \sqrt{|\nabla \bm u_\epsilon|^2 + \epsilon^2} \in \bm L^\infty(\Omega)$ to get
	\begin{align}
		&(\nabla \bm u_\epsilon + g \bm\lambda_\epsilon, \nabla \bm v) - (p_\epsilon, \operatorname{div} \bm v) = (\bm f, \bm v) \qquad \forall \bm v \in \bm V,   \label{eq: proof of thm2.2} \\
		&\operatorname{Tr} \bm\lambda_\epsilon = 0, \quad |\bm\lambda_\epsilon| \le 1 \quad\text{a.e.\ in}\quad \Omega. \notag
	\end{align}
	In view of this and \eref{eq: peps is bounded}, we can choose a sub-sequence of $\{\bm\lambda_\epsilon\}, \{p_\epsilon\}$ (denoted by the same symbols) and some $\bm\lambda \in \bm L^\infty(\Omega), p \in L^2_0(\Omega)$ such that $\operatorname{Tr} \bm\lambda = 0$, $|\bm\lambda| \le 1$ a.e.\ in $\Omega$ and
	\begin{align*}
		\bm\lambda_\epsilon \rightharpoonup \bm\lambda \quad\text{weakly-$*$ in}\quad \bm L^\infty(\Omega), \qquad p_\epsilon \rightharpoonup p \quad\text{weakly in}\quad L^2(\Omega).
	\end{align*}
	Then, letting $\epsilon \to 0$ in \eref{eq: proof of thm2.2} we obtain \eref{eq: VE}.
	Next, taking $\bm v = \bm 0, 2\bm u$ in \eref{eq: VI} we have $\|\nabla\bm u\|_{\bm L^2(\Omega)}^2 + (g, |\nabla\bm u|) = (\bm f, \bm u)$, which combined with \eref{eq: VE} for $\bm v = \bm u$ leads to $(g, |\nabla\bm u|) = (g\bm\lambda, \nabla\bm u)$.
	Since $|\bm\lambda| \le 1$, this holds if and only if $g\bm\lambda : \nabla \bm u = g |\nabla \bm u|$ a.e.\ in $\Omega$.
	Consequently, we have $g\bm\lambda = g\nabla\bm u/|\nabla\bm u|$ where $\nabla\bm u \neq \bm O$.
	Now we find from \eref{eq: VE} that \eref{eq: strong form} holds in the distributional sense (or in $\bm H^{-1}(\Omega)$) and that the second condition in \eref{eq: slipBC} holds in the sense of $\bm H^{-1/2}(\Gamma)$.
\end{proof}
\begin{rem}
	Uniqueness of $\bm\lambda$ (thus that of $p$ neither) may not be true; see \cite[p.\ 354]{GLT1981}.
\end{rem}

\section{$H^2$-regularity for the stationary problem} \label{sec: stationary}
We continue to focus on the stationary problems and to assume $\nu = 1$ as in the previous section.
The following $\bm H^2$-regularity theorem on the regularized problem is the main result of this section.
\begin{thm} \label{thm: H2 estimate for ueps}
	Let $\Omega$ be of $C^{3,1}$-class, $\bm f \in \bm L^2(\Omega)$, $g \in H^1_0(\Omega)$ be non-negative, and $\epsilon \in (0, 1]$.
	Then the solution of $\bm u_\epsilon$ of \eref{eq: VI eps} satisfies $\bm u_\epsilon \in \bm H^2(\Omega)$, $(\nabla \bm u_\epsilon \bm n)_\tau = \bm 0$ a.e.\ on $\Gamma$, and
	\begin{equation*}
		\|\bm u_\epsilon\|_{\bm H^2(\Omega)} \le C (\|\bm f\|_{\bm L^2(\Omega)} + \|g\|_{H^1(\Omega)}),
	\end{equation*}
	where the constant $C = C(d, \Omega)$ is independent of $\epsilon$.
\end{thm}
\begin{rem}
	As for the pressure, although it is possible to show $p_\epsilon \in H^1(\Omega)$, we were not able to get an $H^1(\Omega)$-estimate uniform in $\epsilon$ (see Subsection \ref{sec: nabla' p}).
\end{rem}

This theorem allows us to derive an $\bm H^2$-regularity result for the original variational inequality without regularization.
\begin{cor} \label{cor: main result}
	Under the same assumptions as in \tref{thm: H2 estimate for ueps}, the solution $\bm u$ of \eref{eq: VI} satisfies $\bm u \in \bm H^2(\Omega)$, $(\nabla \bm u \, \bm n)_\tau = \bm 0$ a.e.\ on $\Gamma$, and $\|\bm u\|_{\bm H^2(\Omega)} \le C (\|\bm f\|_{\bm L^2(\Omega)} + \|g\|_{H^1(\Omega)})$.
\end{cor}
\begin{proof}
	There exists a sub-sequence of $\{\bm u_\epsilon\}$ (denoted by the same symbol) and some $\bm u^* \in \bm H^2(\Omega)$ such that $\bm u_\epsilon \rightharpoonup \bm u^*$ weakly in $\bm H^2(\Omega)$ as $\epsilon \to 0$.
	By the compactness, $\bm u_\epsilon$ converges to $\bm u^*$ strongly in $\bm H^1(\Omega)$ and $(\nabla \bm u^* \bm n)_\tau = \bm 0$ a.e.\ on $\Gamma$.
	Now \pref{prop: VI eps,sigma}(ii) and the uniqueness of the limit imply $\bm u = \bm u^*$ (this also implies the $\bm H^2(\Omega)$-weak convergence of the whole sequence $\{\bm u_\epsilon\}$).
\end{proof}

\tref{thm: H2 estimate for ueps} is reduced to the following key proposition, which asserts local $H^2$-regularity up to the boundary and inside the domain:
\begin{prop} \label{main prop}
	Under the same assumptions as in \tref{thm: H2 estimate for ueps}, let $(\bm u_\epsilon, p_\epsilon) \in \bm V_\sigma \times L^2_0(\Omega)$ be as in \pref{prop: VEeps}.
	
	(i) Each $\bm x_0 \in \Gamma$ admits open neighborhoods $W_{\bm x_0}$ and $U_{\bm x_0}$ such that $\overline{W_{\bm x_0}} \subset U_{\bm x_0}$ and $\nabla^2 \bm u_\epsilon \in \bm L^2(\Omega \cap W_{\bm x_0})$.
	Furthermore, there exists a constant $C = C(d, \Omega, \bm x_0, W_{\bm x_0}, U_{\bm x_0})$, independent of $\epsilon$, such that
	\begin{equation*}
		\|\nabla^2 \bm u_\epsilon\|_{\bm L^2(\Omega \cap W_{\bm x_0})} \le C(\|\bm f\|_{\bm L^2(\Omega \cap U_{\bm x_0})} + \|g\|_{H^1(\Omega \cap U_{\bm x_0})}
			+ \|\bm u_\epsilon\|_{\bm H^1(\Omega \cap U_{\bm x_0})} + \|p_\epsilon\|_{L^2(\Omega \cap U_{\bm x_0})}).
	\end{equation*}
	
	(ii) Each $\bm x_0 \in \Omega$ admits open neighborhoods $W_{\bm x_0}$ and $U_{\bm x_0}$ such that $\overline{W_{\bm x_0}} \subset U_{\bm x_0}$,  $\overline{U_{\bm x_0}} \subset \Omega$, and $\nabla^2 \bm u_\epsilon \in \bm L^2(\Omega \cap W_{\bm x_0})$.
	Furthermore, the same estimate as in (i) holds.
\end{prop}

\begin{proof}[Proof of \tref{thm: H2 estimate for ueps}]
	The above proposition combined with a standard covering argument (note that $\overline\Omega$ is compact) and with the estimate $\|\bm u_\epsilon\|_{\bm H^1(\Omega)} + \|p_\epsilon\|_{L^2(\Omega)} \le C(\|\bm f\|_{\bm L^2(\Omega)} + \|g\|_{L^2(\Omega)})$ (recall \pref{prop: VI eps,sigma}(ii) and \eref{eq: peps is bounded}) proves the theorem.
\end{proof}

The rest of this section will be devoted to the proof of \pref{main prop}, especially statement (i).
Statement (ii), i.e., interior regularity may be obtained in a similar (actually easier) manner to (i); see \rref{rem: interior H2} below.

\subsection{Coordinate transformation to flatten the boundary}
The fist step to show \pref{main prop}(i) is, as usual in a regularity theory for elliptic problems, to localize the PDE using a cut-off function and to transform it into a half-space problem.
\subsubsection{Local coordinates} \label{sec: local coordinate}
Since $\Gamma = \partial\Omega$ is of $C^{3,1}$-class, each $\bm x_0 \in \Gamma$ admits the following graph representation based on a hight function $\rho \in C^{3,1}((-2r, 2r)^{d-1})$ with some $r > 0$ after suitable rotation and translation of the $\bm x$-coordinate:
\begin{itemize}
	\item $\bm x_0 = \bm 0$ and $\rho(\bm 0) = 0$
	\item Setting $U(2r, R) := \{\bm x = (\bm x', x_d) \mid \bm x' \in (-2r, 2r)^{d-1}, \, \rho(\bm x') - R < x_d < \rho(\bm x') + R \}$, one has
	\begin{equation*}
		\begin{aligned}
			\Omega \cap U(2r, R) &= \{\rho(\bm x') - R < x_d < \rho(\bm x')\}, \\
			\Omega^c \cap U(2r, R) &= \{\rho(\bm x') < x_d < \rho(\bm x') + R \}, \\
			\Gamma \cap U(2r, R) &= \{x_d = \rho(\bm x') \},
		\end{aligned}
		\qquad\text{for}\quad \bm x' \in (-2r, 2r)^{d-1}.
	\end{equation*}
\end{itemize}
\begin{rem}
	Hereafter we do not explicitly mention the use of rotation and translation because the weak form \eref{eq: VEeps} is invariant under these congruent transformation, cf.\ \cite[p.\ 1096]{BdV04}.
	The ``sizes'' $2r$ and $R$, as well as upper bounds for $\|\rho\|_{C^{3,1}((-2r, 2r)^{d-1})}$, may be taken uniformly for all $\bm x_0 \in \Gamma$.
\end{rem}

In the subsequent discussion, we sometimes require that the last component of a vector be aligned with the normal direction to $\Gamma$.
To this end we introduce the so called ``normal transformation'', cf.\ \cite[Theorem I.2.12]{Wlo1987}, \cite[p.\ 315]{AMS2023}, defined by the relation
\begin{equation*}
	\begin{pmatrix} \bm x' \\ x_d \end{pmatrix} = \begin{pmatrix} \bm y' \\ \rho(\bm y') \end{pmatrix} + y_d \begin{pmatrix} \nabla' \rho(\bm y') \\ -1 \end{pmatrix} =: \bm\psi(\bm y),
\end{equation*}
where $\nabla' = \nabla_{\bm y'}$ denotes the gradient operator w.r.t.\ $\bm y'$.
If $y_d = 0$, the Jacobian matrix $\bm\Psi := \nabla_{\bm y} \bm\psi$ is given by
\begin{equation} \label{eq: Psi and Phi}
	\bm\Psi = \left( \begin{array}{c|c} 
		\bm I & \nabla'\rho \\[1mm] \hline \\[-4mm]
		(\nabla'\rho)^\top & -1 \\
	\end{array} \right), \quad
	\bm\Psi^{-1} = \frac{1}{|\nabla'\rho|^2 + 1} \left( \begin{array}{c|c} 
		(|\nabla'\rho|^2 + 1)\bm I - (\nabla'\rho) (\nabla'\rho)^\top & \nabla'\rho \\[1mm] \hline \\[-4mm]
		(\nabla'\rho)^\top & -1 \\
	\end{array} \right)
	\quad \text{for} \quad y_d = 0.
\end{equation}
By the implicit function theorem, $\bm\psi$ is a $C^{2,1}$-diffeomorphism near $\bm y = \bm 0$, which we may suppose is defined for $\bm y \in (-r, r)^d =: K$.
Setting $\bm\varphi := \bm\psi^{-1}$ thus $\bm y = \bm\varphi(\bm x) \Leftrightarrow \bm x = \bm\psi(\bm y)$, we see that the Jacobian matrix of $\bm\varphi$ is $\bm\Psi^{-1} =: \bm\Phi$, more precisely, $\nabla_{\bm x} \bm\varphi(\bm\psi(\bm y)) = \bm\Phi(\bm y)$.
Moreover, $U_{\bm x_0} := \bm\psi(K)$ is a neighborhood of $\bm x_0$ and 
\begin{equation*}
	\bm\varphi(\Omega \cap U_{\bm x_0}) = K_+ := \{ \bm y \in K \mid y_d > 0 \}, \quad \bm\varphi(\Gamma \cap U_{\bm x_0}) = K_0 := \{\bm y \in K \mid y_d = 0\}.
\end{equation*}

Next we discuss transformation of functions and vectors.
For a scalar function $f$ defined in $U_{\bm x_0}$, its coordinate transformation $\tilde f$ in $K$ is given as $\tilde f(\bm y) = f(\bm x) = f(\bm\psi(\bm y))$, which is nothing but a pullback by $\bm\psi$.
The chain rule reads $\nabla_{\bm x} f = \bm\Phi^\top \nabla_{\bm y} f$, and we have (note that $\|\bm\Phi\|_{\bm L^\infty(K)}, \|\bm\Psi\|_{\bm L^\infty(K)} \le C$)
\begin{equation*}
	|\nabla_{\bm y} \tilde f| = |\bm\Psi^\top \, \bm\Phi^\top \nabla_{\bm y} \tilde f| \le C |\bm\Phi^\top \nabla_{\bm y} \tilde f| \quad\text{for}\quad \bm y \in K.
\end{equation*}
For a vector-valued function $\bm v$ defined in $U_{\bm x_0}$, its pullback to $K$ is again denoted by $\tilde{\bm v} = \bm v \circ \bm\psi$.
The chain rule for this case reads $\nabla_{\bm x} \bm v = (\nabla_{\bm y} \tilde{\bm v}) \bm\Phi$, where we recall that the gradient of a vector is defined as
\begin{equation*}
	(\nabla_{\bm x} \bm v)_{ij} = \frac{\partial v_i}{\partial x_j} \quad (i, j = 1, \dots, d), \quad\text{that is,}\quad \nabla_{\bm x} \bm v = (\nabla_{\bm x} \otimes \bm v)^\top.
\end{equation*}
Hereafter we do not indicate the subscripts $\bm x$ or $\bm y$ in the gradient operator explicitly when no confusion occurs.

To describe the outer unit normal $\bm n$ to $\Gamma$ in local coordinates, let us define
\begin{equation*}
	\bm N(\bm y') = \frac{1}{\sqrt{|\nabla'\rho(\bm y')|^2 + 1}} \begin{pmatrix} -\nabla' \rho(\bm y') \\ 1 \end{pmatrix} = \bm n(\bm\psi(\bm y', 0)) \quad\text{for}\quad (\bm y', 0) \in K_0.
\end{equation*}
Considering a pullback $\tilde{\bm v} = (\tilde v_i)_{1\le i\le d}$ only is not convenient in some cases because its components are not aligned with $\bm N$ in general.
Thereby we introduce a ``transformation of frame'' by
\begin{equation*}
	\underline{\tilde{\bm v}} = \mathring{\bm\Psi}^{-1} \tilde{\bm v} \Longleftrightarrow \tilde{\bm v} = \mathring{\bm\Psi} \underline{\tilde{\bm v}} \quad\text{where}\quad \mathring{\bm\Psi} := \bm\Psi|_{y_d = 0} \quad \text{(cf.\ \eref{eq: Psi and Phi})}.
\end{equation*}
On $K_0$, we simply write $\bm\Psi$ instead of $\mathring{\bm\Psi}$.
It is clear that the $\bm H^m(K)$-norms of $\tilde{\bm v}$ and $\underline{\tilde{\bm v}}$ are equivalent to each other for $m = 0, 1, \dots$ (up to multiplication by a constant depending only on $\Omega$).
Other properties of this new frame are collected as follows.

\begin{lem} \label{lem: properties of new frame}
	With the notation above we have:
	
	(i) $\bm\Phi \bm N = (|\nabla'\rho|^2 + 1)^{-1/2} \underline{\bm n}$ and $\bm\Psi \bm N = (|\nabla'\rho|^2 + 1)^{1/2} \underline{\bm n}$ on $K_0$, where $\underline{\bm n} := (0, \dots, 0, -1)^\top$ denotes the outer unit normal to $K_0$.
	
	(ii) $\bm v \cdot \bm n = 0$ on $\Gamma \cap U_{\bm x_0} \Longleftrightarrow \underline{\tilde v}_d = 0$ on $K_0$.
	
	(iii) $(\nabla \bm v \, \bm n)_\tau = \bm 0$ on $\Gamma \cap U_{\bm x_0} \Longleftrightarrow \partial_d \underline{\tilde{\bm v}}' = \bm 0$ on $K_0$, where $\underline{\tilde{\bm v}}' := (\underline{\tilde v}_1, \dots, \underline{\tilde v}_{d-1})^\top$.

	(iv) If $\underline{\tilde v}_d = 0$ on $K_0$, then $\underline{\tilde{\bm v}}' = \tilde{\bm v}' := (\tilde v_1, \dots, \tilde v_{d-1})^\top$ on $K_0$.
\end{lem}
\begin{proof}
	(i)(iv) These follow from a direct computation using \eref{eq: Psi and Phi}.
	
	(ii) By (i) and symmetry of the matrix $\bm\Psi$, the left-hand side in the $\bm y$-coordinate reads
	\begin{equation*}
		\bm N^\top \tilde{\bm v} = 0 \Longleftrightarrow \bm N^\top \bm\Psi \underline{\tilde{\bm v}} = 0 \Longleftrightarrow -(|\nabla'\rho|^2 + 1)^{1/2} \underline{\tilde v}_d = 0.
	\end{equation*}
	
	(iii) By a direct computation in the $\bm y$-coordinate and transformed frame (note that $\partial_d \mathring{\bm\Psi} = \bm O$),
	\begin{align*}
		(\nabla \bm v \, \bm n)_\tau &= (\bm I - \bm N \bm N^\top) \nabla (\mathring{\bm\Psi} \underline{\tilde{\bm v}}) \bm\Phi \bm N
				= (|\nabla'\rho|^2 + 1)^{-1/2} (\bm I - \bm N \bm N^\top) \nabla (\mathring{\bm\Psi} \underline{\tilde{\bm v}}) \underline{\bm n} \\
			&= -(|\nabla'\rho|^2 + 1)^{-1/2} (\bm I - \bm N \bm N^\top) \bm\Psi \partial_d \underline{\tilde{\bm v}}
				= - \frac{1}{\sqrt{|\nabla'\rho|^2 + 1}} \left[ \bm\Psi - \begin{pmatrix} -\nabla'\rho \\ 1 \end{pmatrix} \begin{pmatrix} 0 & -1 \end{pmatrix} \right] \partial_d \underline{\tilde{\bm v}} \\
			&=- \frac{1}{\sqrt{|\nabla'\rho|^2 + 1}} \begin{pmatrix} \partial_d \underline{\tilde{\bm v}}' \\ (\nabla'\rho)^\top \partial_d \underline{\tilde{\bm v}}' \end{pmatrix},
	\end{align*}
	which yields the desired equality.
\end{proof}

\subsubsection{Reduction to a half-space problem}
For simplicity in the notation, in the subsequent sections, we agree to omit the subscript $\epsilon$ from the solution $(\bm u_\epsilon, p_\epsilon)$ of \eref{eq: VEeps}.
Thus $(\bm u, p) \in \bm V \times L^2_0(\Omega)$ satisfies $\partial_i u_i = 0$ and
\begin{equation*}
	\int_\Omega \partial_j u_i \, \partial_j v_i \, d\bm x + \int_\Omega g \frac{\partial_j u_i \, \partial_j v_i}{\sqrt{|\nabla \bm u|^2 + \epsilon^2}} \, d\bm x
	 	- \int_\Omega p \, \partial_i v_i \, d\bm x = \int_\Omega f_i v_i \, d\bm x \qquad \forall \bm v \in \bm V,
\end{equation*}
where we exploit Einstein's summation convention, omitting the symbols $\sum_{i,j=1}^d$ and $\sum_{i=1}^d$.

For fixed $\bm x_0 \in \Gamma$, we consider test functions $\bm v \in \bm V$ with $\operatorname{supp} \bm v \subset U_{\bm x_0}$ and transform the above weak formulation using  $\bm y = \bm\varphi(\bm x)$ together with 
$\partial/\partial x_j = \Phi_{kj}\partial/\partial y_k$ (the summation convention is applied to $k$ as well) and denoting the Jacobian determinant by $J := \operatorname{det} \bm\Psi$ to obtain
\begin{align*}
	&\int_{K_+} \partial_k \tilde u_i \Phi_{kj} \, \partial_l \tilde v_i \Phi_{lj} \, |J| \, d\bm y
	+ \int_{K_+} \tilde g \, \frac{ \partial_k \tilde u_i \Phi_{kj} \, \partial_l \tilde v_i \Phi_{lj} }{ ( |(\nabla\tilde{\bm u})\bm\Phi|^2 + \epsilon^2 )^{1/2} } \, |J| \, d\bm y
	- \int_{K_+} \tilde p \, \partial_k \tilde v_i \Phi_{ki} \, |J| \, d\bm y \\
	&= \int_{K_+} \tilde f_i \tilde v_i \, |J| \, d\bm y
	\qquad \forall \tilde{\bm v} \in \bm H^1(K_+), \; \tilde{\bm v} \cdot \bm N = 0 \text{ on } K_0, \; \operatorname{supp} \tilde{\bm v} \subset K, \\
	&\partial_k \tilde u_i \Phi_{ki} = 0 \quad\text{in } K_+.
\end{align*}
Regarding $\tilde{\bm v} |J|$ as a new test function, we see that
\begin{equation} \label{eq: weak form in y}
\begin{aligned}
	&\int_{K_+} \partial_k \tilde u_i \Phi_{kj} \, \partial_l \tilde v_i \Phi_{lj} \, d\bm y
	+ \int_{K_+} \tilde g \, \frac{ \partial_k \tilde u_i \Phi_{kj} \, \partial_l \tilde v_i \Phi_{lj} }{ ( |(\nabla\tilde{\bm u})\bm\Phi|^2 + \epsilon^2 )^{1/2} } \, d\bm y
	- \int_{K_+} \tilde p \, \partial_k \tilde v_i \Phi_{ki} \, d\bm y \\
	&= \int_{K_+} \tilde F_i \tilde v_i \, d\bm y \qquad
	\forall \tilde{\bm v} \in \bm H^1(K_+), \; \tilde{\bm v} \cdot \bm N = 0 \text{ on } K_0, \; \operatorname{supp} \tilde{\bm v} \subset K.
\end{aligned}
\end{equation}
Here,
\begin{equation*}
	\tilde F_i := \tilde f_i - \partial_k \tilde u_i \Phi_{kj} \, \Phi_{lj} \frac{\partial_l |J|}{|J|}
		- \tilde g \, \frac{ \partial_k \tilde u_i \Phi_{kj} \, \Phi_{lj} \frac{\partial_l |J|}{|J|} }{ ( |(\nabla\tilde{\bm u}) \bm\Phi|^2 + \epsilon^2 )^{1/2} }
		+ \tilde p \, \Phi_{ki} \frac{\partial_k |J|}{|J|} \quad (i = 1, \dots, d)
\end{equation*}
satisfies, because $ \partial_k \tilde u_i \Phi_{kj} / ( |(\nabla\tilde{\bm u}) \bm\Phi|^2 + \epsilon^2 )^{1/2}) \le 1$,
\begin{equation} \label{eq: tilde F}
	\|\tilde{\bm F}\|_{\bm L^2(K_+)} \le C (\|\tilde{\bm f}\|_{\bm L^2(K_+)} + \|\tilde g\|_{L^2(K_+)} + \|\tilde{\bm u}\|_{\bm H^1(K_+)} + \|\tilde p\|_{L^2(K_+)}  ).
\end{equation}

We further transform the frame by $\tilde u_i = \mathring\Psi_{ia} \underline{\tilde u}_a$ and $\tilde v_i = \mathring\Psi_{ib} \underline{\tilde v}_b$ to arrive at
\begin{equation} \label{eq2: weak form in y}
\begin{aligned}
	&\int_{K_+} \partial_k (\mathring\Psi_{ia} \underline{\tilde u}_a) \Phi_{kj} \, \partial_l (\mathring\Psi_{ib} \underline{\tilde v}_b) \Phi_{lj} \, d\bm y
	+ \int_{K_+} \tilde g \, \frac{ \partial_k (\mathring\Psi_{ia} \underline{\tilde u}_a) \Phi_{kj} \, \partial_l (\mathring\Psi_{ib} \underline{\tilde v}_b) \Phi_{lj} }{ ( |\nabla (\mathring{\bm\Psi} \underline{\tilde{\bm u}}) \bm\Phi|^2 + \epsilon^2 )^{1/2} } \, d\bm y
	- \int_{K_+} \tilde p \, \partial_k (\mathring\Psi_{ib} \underline{\tilde v}_b) \Phi_{ki} \, d\bm y \\
	&= \int_{K_+} \tilde F_i \mathring\Psi_{ib} \underline{\tilde v}_b \, d\bm y \qquad
	\forall \tilde{\bm v} \in \bm H^1(K_+), \; \underline{\tilde v}_d = 0 \text{ on } K_0, \; \operatorname{supp} \tilde{\bm v} \subset K,
\end{aligned}
\end{equation}
where \lref{lem: properties of new frame}(ii) has been used.

\subsection{Regularity in the tangential direction}
The next step to show \pref{main prop}(i) is to establish regularity w.r.t.\ $\bm y'$ for the solution $(\tilde{\bm u}, \tilde p)$ of \eref{eq: weak form in y}.
If we set $W := (-r/2, r/2)^d$, then $W_{\bm x_0} := \bm\psi(W)$ is an open neighborhood of $\bm x_0$ and $\overline{W_{\bm x_0}} \subset U_{\bm x_0}$.
Hence there exists a cut-off function $\theta \in C_0^\infty(K)$ such that $\theta \equiv 1$ in $\overline{W}$ and $0\le \theta \le1$ in $\mathbb R^d$.
Tangential regularity of $\tilde{\bm u}$ and that of $\tilde p$ are addressed below.
\subsubsection{Estimates for $\nabla\nabla' \tilde{\bm u}$ uniform in $\epsilon$}
\begin{lem} \label{lem: tangential direction}
	We have $\nabla \nabla' \tilde{\bm u} \in \bm L^2(K_+ \cap W)$ and
	\begin{equation*}
		\|\nabla \nabla' \tilde{\bm u}\|_{\bm L^2(K_+ \cap W)} \le C(\|\bm f\|_{\bm L^2(\Omega \cap U_{\bm x_0})} + \|g\|_{H^1(\Omega \cap U_{\bm x_0})}
			+ \|\bm u\|_{\bm H^1(\Omega \cap U_{\bm x_0})} + \|p\|_{L^2(\Omega \cap U_{\bm x_0})}).
	\end{equation*}
\end{lem}

The proof relies on the method of difference quotient, notation of which is introduced as follows.
The standard basis of $\mathbb R^d$ is denoted by $\bm e_1, \dots, \bm e_d$.
For $\alpha = 1, \dots, d-1$ and $0 < |h| < (1/2) \operatorname{dist}(\operatorname{supp}\theta, \partial K)$, the shift and difference quotient operators in the direction $\bm e_\alpha$ are defined by
\begin{equation*}
	(s_h^\alpha f)(\bm y) = f(\bm y + h\bm e_\alpha), \qquad D_h^\alpha f = \frac{s_h^\alpha f - f}{h}.
\end{equation*}
Since it is clear that one direction $\alpha$ is fixed, we omit the superscript $\alpha$ below.
The following formulas are well known (see e.g.  \cite[Chapter 9]{Bre2010}):
\begin{align*}
	D_h(fg) &= f (D_hg) + (D_h f) (s_hg), \\
	\int_{K_+} f \, (D_{-h} g) \, d\bm y &= \int_{K_+} (D_h f) \, g \, d\bm y \quad\text{if } f, g \in L^2(K_+) \text{ and } \operatorname{supp} f \cup \operatorname{supp} g \subset \operatorname{supp}\theta \\
	\|D_h f\|_{L^p(K_+)} &\le \|\nabla f\|_{L^p(K_+)} \quad (1\le p\le \infty) \quad\text{if } f \in W^{1, p}(K_+), \; \operatorname{supp} f \subset \operatorname{supp}\theta.
\end{align*}
\begin{proof}[Proof of \lref{lem: tangential direction}]
	Take $\underline{\tilde{\bm v}}= D_{-h}(\theta^2 D_h \underline{\tilde{\bm u}})$ in \eref{eq2: weak form in y} (note that this is a legitimate test function) to get
	\begin{align*}
		&\int_{K_+} \partial_k (\mathring\Psi_{ia} \underline{\tilde u}_a) \Phi_{kj} \, \partial_l \big( \mathring\Psi_{ib} D_{-h}(\theta^2 D_h \underline{\tilde u}_b) \big) \Phi_{lj} \, d\bm y
		+ \int_{K_+} \tilde g \, \frac{ \partial_k (\mathring\Psi_{ia} \underline{\tilde u}_a) \Phi_{kj} \, \partial_l \big( \mathring\Psi_{ib} D_{-h}(\theta^2 D_h \underline{\tilde u}_b) \big) \Phi_{lj} }{ ( |\nabla (\mathring{\bm\Psi} \underline{\tilde{\bm u}}) \bm\Phi|^2 + \epsilon^2 )^{1/2} } \, d\bm y \\
		&\hspace{1cm} - \int_{K_+} \tilde p \, \partial_k \big( \mathring\Psi_{ib} D_{-h}(\theta^2 D_h \underline{\tilde u}_b) \big) \Phi_{ki} \, d\bm y
			= \int_{K_+} \tilde F_i \mathring\Psi_{ib} D_{-h}(\theta^2 D_h \underline{\tilde u}_b) \, d\bm y,
	\end{align*}
	which is rephrased as
	\begin{equation*}
		J_1 + J_2 + J_3 = J_4.
	\end{equation*}
	
	We divide each of $J_1$ to $J_4$ into ``principal'' and ``lower-order'' parts in terms of the number of differentiation or difference quotient applied to $\tilde{\bm u}$.
	For example, for $J_1$ define $M_1$ by the following equality:
	\begin{align*}
		J_1 &= \int_{K_+} \theta \partial_k D_h (\mathring\Psi_{ia} \underline{\tilde u}_a) \Phi_{kj} \cdot \theta \partial_l D_h(\mathring\Psi_{ib} \underline{\tilde u}_b) \Phi_{lj} \, d\bm y + M_1.
	\end{align*}
	Then we find that each term of $M_1$ involves two derivatives and two difference quotients and that at least one of them is applied to $\mathring{\bm\Psi}, \bm\Phi, \theta$---namely quantities other than $\tilde{\bm u}$.
	This observation implies
	\begin{equation} \label{eq: estimate of M1}
		|M_1| \le C \|\tilde{\bm u}\|_{\bm H^1(K_+)} (\|\tilde{\bm u}\|_{\bm H^1(K_+)} + \|\theta \, \nabla D_h \tilde{\bm u}\|_{\bm L^2(K_+)}),
	\end{equation}
	where we go back from the transformed frame $\underline{\tilde{\bm u}}$ to the original one $\tilde{\bm u}$ using an equivalence relation $|\nabla \underline{\tilde{\bm u}}| \le C(|\nabla \tilde{\bm u}| + |\tilde{\bm u}|)$ a.e.\ on $K_+$, etc.
	Therefore,
	\begin{align}
		J_1 &= \int_{K_+} \theta^2 |(\nabla D_h \tilde{\bm u}) \bm\Phi|^2 \, d\bm y + M_1 \ge C \|\theta \, \nabla D_h \tilde{\bm u}\|_{\bm L^2(K_+)}^2 + M_1, \label{eq1: proof of tangential direction}
	\end{align}
	with the lower-order terms $M_1$ satisfying \eref{eq: estimate of M1}.
	In this way, we agree to omit the detailed expressions of the lower-order terms, only displaying their bounds based on inequalities.
	Then, for $J_2, J_3, J_4$ we have
	{\allowdisplaybreaks
	\begin{align}
		J_2 &= \int_{K_+} \theta^2 \tilde g \, D_h \bigg[ \frac{\partial_k \tilde u_i}{( |(\nabla\tilde{\bm u})\bm\Phi|^2 + \epsilon^2 )^{1/2}} \bigg] \Phi_{kj} \cdot
			\partial_l (D_h \tilde u_i) \Phi_{lj} \, d\bm y + M_2 =: J_{21} + M_2, \notag \\
			&\hspace{3cm} |M_2| \le C\|\tilde g\|_{H^1(K_+)} (\|\tilde{\bm u}\|_{\bm H^1(K_+)} + \|\theta \, \nabla D_h \tilde{\bm u}\|_{\bm L^2(K_+)}), \notag \\[2mm]
		J_3 &= - \int_{K_+} \theta^2 \tilde p \, D_{-h} D_h (\underbrace{ \partial_k \tilde u_i \Phi_{ki} }_{=0}) \, d\bm y + M_3 = M_3, \notag \\
			&\hspace{3cm} |M_3| \le C\|\tilde p\|_{L^2(K_+)} (\|\tilde{\bm u}\|_{\bm H^1(K_+)} + \|\theta \, \nabla D_h \tilde{\bm u}\|_{\bm L^2(K_+)}), \notag \\[2mm]
		J_4 &= \int_{K_+} \theta^2 \tilde F_i (D_{-h} D_h \tilde u_i) \, d\bm y + M_4 \le \|\tilde{\bm F}\|_{\bm L^2(K_+)} \|\theta \, \nabla D_h \tilde{\bm u}\|_{\bm L^2(K_+)} + M_4, \notag \\
			&\hspace{3cm} |M_4| \le \|\tilde{\bm F}\|_{\bm L^2(K_+)} \|\tilde{\bm u}\|_{\bm H^1(K_+)}, \notag
	\end{align}}
	For $J_{21}$, a direct computation using 
$1/\sqrt a - 1/\sqrt b = (b - a)/(\sqrt a \sqrt b (\sqrt a + \sqrt b))$ shows
	\begin{align*}
		&D_h \bigg[ \frac{\partial_k \tilde u_i}{( |(\nabla\tilde{\bm u})\bm\Phi|^2 + \epsilon^2 )^{1/2}} \bigg] \Phi_{kj}
		= \frac{\partial_k (D_h\tilde u_i) \Phi_{kj}}{s_h(|(\nabla\tilde{\bm u})\bm\Phi|^2 + \epsilon^2)^{1/2}} \\
		&\quad - \frac{D_h((\nabla\tilde{\bm u})\bm\Phi)}{s_h(|(\nabla\tilde{\bm u})\bm\Phi|^2 + \epsilon^2)^{1/2}} :
		\frac{s_h((\nabla\tilde{\bm u})\bm\Phi) + (\nabla\tilde{\bm u})\bm\Phi}{s_h(|(\nabla\tilde{\bm u})\bm\Phi|^2 + \epsilon^2)^{1/2} + (|(\nabla\tilde{\bm u})\bm\Phi|^2 + \epsilon^2)^{1/2}} \,
		\frac{\partial_k \tilde u_i \Phi_{kj}}{(|(\nabla\tilde{\bm u})\bm\Phi|^2 + \epsilon^2)^{1/2}}.
	\end{align*}
	We multiply this by $\theta^2 \tilde g \cdot \partial_l (D_h \tilde u_i) \Phi_{lj}$ and add the resulting equations for $i, j = 1, \dots, d$.
	Then, as a result of
	\begin{equation*}
		\frac{D_h((\nabla\tilde{\bm u})\bm\Phi)}{s_h(|(\nabla\tilde{\bm u}) \bm\Phi|^2 + \epsilon^2)^{1/2}}
		= \frac{\nabla (D_h\tilde{\bm u}) \bm\Phi}{s_h(|(\nabla\tilde{\bm u})\bm\Phi|^2 + \epsilon^2)^{1/2}}
		+ \underbrace{s_h \bigg[ \frac{\nabla\tilde{\bm u}}{(|(\nabla\tilde{\bm u})\bm\Phi|^2 + \epsilon^2)^{1/2}} \bigg] }_{ |\cdot| \le C} D_h\bm\Phi
	\end{equation*}
	and the Cauchy--Schwarz inequality $|\bm P|^2 |\bm Q| |\bm R| \ge (\bm P : \bm Q)(\bm R : \bm P)$, we deduce that
	\begin{align*}
		J_{21} &= \int_{K+} \theta^2 \tilde g \frac{
				|\bm B|^2 \cdot (s_h + \operatorname{id})(|\bm A|^2 + \epsilon^2)^{1/2} \cdot (|\bm A|^2 + \epsilon^2)^{1/2} - \big[ \bm B : (s_h + \operatorname{id})\bm A \big] \big[ \bm A : \bm B \big]
			}{
				s_h(|\bm A|^2 + \epsilon^2)^{1/2} \cdot (s_h + \operatorname{id})(|\bm A|^2 + \epsilon^2)^{1/2} \cdot (|\bm A|^2 + \epsilon^2)^{1/2}
			} \, d\bm y + M_{21} \ge 0 + M_{21}, \\
		&\hspace{3cm}
			|M_{21}| \le C \|\tilde g\|_{L^2(K_+)} \|\theta \, \nabla D_h \tilde{\bm u}\|_{\bm L^2(K_+)},
	\end{align*}
	where $\operatorname{id}$ means an identity operator and we set $\bm A := (\nabla\tilde{\bm u})\bm\Phi$ and $\bm B := \nabla(D_h\tilde{\bm u}) \bm\Phi$.
	
	Collecting the above estimates, applying the arithmetic-geometric inequality, and absorbing $\|\theta \, \nabla D_h \tilde{\bm u}\|_{\bm L^2(K_+)}^2$ to the right-hand side of \eref{eq1: proof of tangential direction}, we arrive at
	\begin{equation} \label{eq3: proof of tangential direction}
		\|\theta \, \nabla D_h \tilde{\bm u}\|_{\bm L^2(K_+)}^2
		\le C (\|\tilde{\bm f}\|_{\bm L^2(K_+)}^2 + \|\tilde g\|_{H^1(K_+)}^2 + \|\tilde{\bm u}\|_{\bm H^1(K_+)}^2 + \|\tilde p\|_{L^2(K_+)}^2),
	\end{equation}
	which implies
	\begin{equation*}
		\|\nabla D_h \tilde{\bm u}\|_{\bm L^2(K_+ \cap W)}
			\le C (\|\tilde{\bm f}\|_{\bm L^2(K_+)} + \|\tilde g\|_{H^1(K_+)} + \|\tilde{\bm u}\|_{\bm H^1(K_+)} + \|\tilde p\|_{L^2(K_+)}).
	\end{equation*}
	Recalling that $D_h$ means $D_h^\alpha$ where $\alpha = 1, \dots, d-1$ is arbitrarily fixed, we take the limit $h \to 0$ to conclude $\nabla \nabla' \tilde{\bm u} \in \bm L^2(K_+ \cap W)$ and
	\begin{equation} \label{eq2: proof of tangential direction}
		\|\nabla \nabla' \tilde{\bm u}\|_{\bm L^2(K_+ \cap W)} \le C (\|\tilde{\bm f}\|_{\bm L^2(K_+)} + \|\tilde g\|_{H^1(K_+)} + \|\tilde{\bm u}\|_{\bm H^1(K_+)} + \|\tilde p\|_{L^2(K_+)}).
	\end{equation}
	This completes the proof of \lref{lem: tangential direction}.
\end{proof}

\subsubsection{Estimates for $\nabla' \tilde p$ (not uniform in $\epsilon$)} \label{sec: nabla' p}
Let us state an inf-sup condition adapted to the $\bm y$-coordinate.
\begin{lem} \label{lem: inf-sup in y}
	There exist positive constants $C_1$ and $C_2$ depending only on $\Omega$ and $U_{\bm x_0}$ such that
	\begin{equation*}
		C_1 \|\tilde q\|_{L^2(K_+)} \le \sup_{\tilde{\bm v} \in \bm H^1_0(K_+)} \frac{ \int_{K_+} \tilde q \, \partial_k \tilde v_i \Phi_{ki} \, d\bm y }{ \|\tilde{\bm v}\|_{\bm H^1(K_+)} } + C_2 \bigg| \int_{K_+} \tilde q \, d\bm y \bigg| \qquad \forall \tilde q \in L^2(K_+).
	\end{equation*}
\end{lem}
\begin{proof}
	Given $\tilde q \in L^2(K_+)$, we start from an inf-sup condition associated with the Lipschitz domain $\Omega \cap U_{\bm x_0}$ (see \cite[Lemma I.4.1(ii)]{GiRa1986}) in the $\bm x$-coordinate:
	\begin{equation*}
		C \bigg\| q - \frac1{|\Omega \cap U_{\bm x_0}|} \int_{\Omega \cap U_{\bm x_0}} q \, d\bm x \bigg\|_{L^2(\Omega)}
			\le \sup_{\bm v \in \bm H^1_0(\Omega \cap U_{\bm x_0})} \frac{ \int_{\Omega \cap U_{\bm x_0}} q \operatorname{div} \bm v \, d\bm x }{ \|\bm v\|_{\bm H^1(\Omega \cap U_{\bm x_0})} },
	\end{equation*}
	where $|\cdot|$ means the $d$-dimensional measure.
	Obviously,
	\begin{equation*}
		C \|q\|_{L^2(\Omega)} 
			\le \sup_{\bm v \in \bm H^1_0(\Omega \cap U_{\bm x_0})} \frac{ \int_{\Omega \cap U_{\bm x_0}} q \operatorname{div} \bm v \, d\bm x }{ \|\bm v\|_{\bm H^1(\Omega \cap U_{\bm x_0})} }
			+ C \bigg| \int_{\Omega \cap U_{\bm x_0}} q \, d\bm x \bigg|.
	\end{equation*}	
	Transforming this into the $\bm y$-coordinate, we obtain
	\begin{equation*}
		C \bigg( \int_{K_+} |\tilde q|^2 |J| \, d\bm y \bigg)^{1/2}
			\le \sup_{\tilde{\bm v} \in \bm H^1_0(K_+)} \frac{ \int_{K_+} \tilde q \, \partial_k \tilde v_i \Phi_{ki} |J| \, d\bm y }{ \|\tilde{\bm v}\|_{\bm H^1(K_+)} } + C \bigg| \int_{K_+} \tilde q |J| \, d\bm y \bigg|.
	\end{equation*}
	Regarding $\tilde q |J|$ as new $\tilde q$ and using $C \le 1/|J|$, we conclude the desired estimate.
\end{proof}

Let $\tilde{\bm v} \in \bm H^1_0(K_+)$ be arbitrary and $0< |h| < (1/2) \operatorname{dist}(\operatorname{supp}\theta, \partial K)$.
We take $D_{-h}(\theta \underline{\tilde{\bm v}})$ as a test function in \eref{eq2: weak form in y} and then return to the original frame $\tilde{\bm u}, \tilde{\bm v}$ to have
\begin{align*}
	\int_{K_+} (\theta D_h \tilde p) \, \partial_k \tilde v_i \Phi_{ki} \, d\bm y
	&= \int_{K_+} \theta (\partial_k D_h\tilde u_i) \Phi_{kj} \, \partial_l \tilde v_i \Phi_{lj} \, d\bm y
		+ \int_{K_+} \theta \tilde g \, D_h \Big[ \frac{\partial_k \tilde u_i}{( |(\nabla\tilde{\bm u})\bm\Phi|^2 + \epsilon^2 )^{1/2}} \Big] \Phi_{kj} \, \partial_l \tilde v_i \Phi_{lj} \, d\bm y \\
	&\hspace{3cm} - \int_{K_+} \tilde F_i \, D_{-h} (\theta \tilde v_i) \, d\bm y + N_1,
\end{align*}
where, as in the proof of \lref{lem: tangential direction}, only principal parts are explicitly shown and all the other lower-order terms are represented by $N_1$.
We see that $N_1$ is bounded as
\begin{equation*}
	|N_1| \le C (\|\tilde g\|_{H^1(K_+)} + \|\tilde{\bm u}\|_{\bm H^1(K_+)} + \|\tilde p\|_{L^2(K_+)}) \|\tilde{\bm v}\|_{\bm H^1(K_+)}.
\end{equation*}
We also notice that
\begin{equation*}
	\left| D_h \Big[ \frac{\partial_k \tilde u_i}{( |(\nabla\tilde{\bm u})\bm\Phi|^2 + \epsilon^2 )^{1/2}} \Big] \Phi_{kj} \right|
	\le C(1 + \epsilon^{-1}) |\nabla D_h\tilde{\bm u}| \quad\text{a.e.\ in } K_+ \quad (i, j = 1, \dots, d).
\end{equation*}
Consequently,
\begin{align*}
	\frac{\displaystyle \bigg| \int_{K_+} (\theta D_h \tilde p) \, \partial_k \tilde v_i \Phi_{ki} \, d\bm y \bigg| }{ \|\tilde{\bm v}\|_{\bm H^1(K_+)} }
	&\le C(1 + \epsilon^{-1}) (\|\tilde{\bm f}\|_{\bm L^2(K_+)} + \|\tilde g\|_{H^1(K_+)} + \|\tilde{\bm u}\|_{\bm H^1(K_+)} + \|\tilde p\|_{L^2(K_+)}
		+ \|\theta \nabla D_h \tilde{\bm u}\|_{\bm L^2(K_+)}) \\
	&\le C(1 + \epsilon^{-1}) (\|\tilde{\bm f}\|_{\bm L^2(K_+)} + \|\tilde g\|_{H^1(K_+)} + \|\tilde{\bm u}\|_{\bm H^1(K_+)} + \|\tilde p\|_{L^2(K_+)}),
\end{align*}
where we have used \eref{eq2: proof of tangential direction}.
Substituting this into \lref{lem: inf-sup in y} with $\tilde q = \theta D_h \tilde p$ and noting that $|\int_{K_+} \theta \, (D_h\tilde p) \, d\bm y| = |\int_{K_+} (D_{-h}\theta) \, \tilde p \, d\bm y| \le C \|\tilde p\|_{L^2(K_+)}$, we deduce
\begin{equation*}
	\|\theta D_h\tilde p\|_{L^2(K_+)} \le C(1 + \epsilon^{-1}) (\|\tilde{\bm f}\|_{\bm L^2(K_+)} + \|\tilde g\|_{H^1(K_+)} + \|\tilde{\bm u}\|_{\bm H^1(K_+)} + \|\tilde p\|_{L^2(K_+)}).
\end{equation*}
Restricting to $K_{+} \cap W$ and making $h \to 0$ prove $\nabla' \tilde p \in \bm L^2(K_{+} \cap W)$.

\begin{rem} \label{rem: interior H2}
	If $\bm x_0 \in \Omega$ and its neighborhoods $W_{\bm x_0}, U_{\bm x_0}$ are chosen as in \pref{main prop}(ii), then the above difference-quotient argument works as well for all the directions $\alpha = 1, \dots, d$ (and for $|h|$ sufficiently small) without transforming the domain to a half-space (in other words, one may assume $\bm\varphi$ is identity).
	Therefore, \pref{main prop}(ii), that is, the interior regularities $\bm u \in \bm H^2_\text{loc}(\Omega)$ and $p \in H^1_\text{loc}(\Omega)$ are already proved; cf.\ \cite[Theorem 3.3.1]{FuSe2000}.
	In particular, $\nabla^2\bm u$ and $\nabla p$ make sense as measurable functions defined in $\Omega$, and the strong form of the PDE holds in the almost-everywhere sense:
	\begin{equation*}
		-\Delta \bm u - \operatorname{div} \Big( g \frac{\nabla \bm u}{\sqrt{|\nabla \bm u|^2 + \epsilon^2}} \Big) + \nabla p = \bm f \quad\text{a.e.\ in } \Omega.
	\end{equation*}
\end{rem}

\subsection{Regularity in the normal direction}
It remains to examine the regularity of $\tilde{\bm u}$ w.r.t.\ $y_d$.
To this end, we re-choose $W$ and $\theta$ in such a way that $W := (-r/4, r/4)^d$ and $\theta \in C_0^\infty((-r/2, r/2)^d)$, $\theta \equiv 1$ in $\overline{W}$ and $0\le \theta \le1$ in $\mathbb R^d$.
The following is a key lemma of this subsection.
\begin{lem} \label{lem: normal direction}
	We have
	\begin{equation*}
		\|\nabla\partial_d \tilde{\bm u}\|_{\bm L^2(K_+ \cap W)} \le C(\|\bm f\|_{\bm L^2(\Omega \cap U_{\bm x_0})} + \|g\|_{H^1(\Omega \cap U_{\bm x_0})}
			+ \|\bm u\|_{\bm H^1(\Omega \cap U_{\bm x_0})} + \|p\|_{L^2(\Omega \cap U_{\bm x_0})}).
	\end{equation*}
\end{lem}

Once this is proved, Lemmas \ref{lem: tangential direction} and \ref{lem: normal direction}, transformed back to the $\bm x$-coordinate, imply
\begin{equation*}
	\|\nabla^2 \bm u\|_{\bm L^2(\Omega \cap W_{\bm x_0})} \le C(\|\bm f\|_{\bm L^2(\Omega \cap U_{\bm x_0})} + \|g\|_{H^1(\Omega \cap U_{\bm x_0})}
			+ \|\bm u\|_{\bm H^1(\Omega \cap U_{\bm x_0})} + \|p\|_{L^2(\Omega \cap U_{\bm x_0})}),
\end{equation*}
which completes the proof of \pref{main prop}(i).
Now we show \lref{lem: normal direction} via the following four steps.

\subsubsection{Proof of $\partial_d^2 \tilde{\bm u} \in \bm L^2(K_+)$ and $\partial_d \tilde p \in L^2(K_+)$}
Let us first prove $\bm u \in \bm H^2(\Omega)$ and $p \in H^1(\Omega)$ up to the boundary, allowing the estimates to depend on $\epsilon$.
Observe that the strong form of \eref{eq: weak form in y} is as follows:
\begin{equation} \label{eq: strong form in y}
\begin{aligned}
	- \partial_l \big( \partial_k \tilde u_i \Phi_{kj} \Phi_{lj} \big)
	- \partial_l \bigg[ \tilde g \, \frac{ \partial_k \tilde u_i \Phi_{kj} }{ ( |(\nabla\tilde{\bm u})\bm\Phi|^2 + \epsilon^2 )^{1/2} } \Phi_{lj} \bigg]
	+ \partial_k (\tilde p \, \Phi_{ki} ) &= \tilde F_i \quad\text{in}\; K_+ \quad (i = 1, \dots, d), \\
	\partial_k \tilde u_i \Phi_{ki} &= 0 \quad\text{in}\; K_+,
\end{aligned}
\end{equation}
which, as noticed in \rref{rem: interior H2}, holds in the almost-everywhere sense.
Focusing on the principal terms, we have
\begin{equation*}
	- \Phi_{kj} \Phi_{lj} \partial_k\partial_l \tilde u_i
	- \tilde g \, \Phi_{kj} \Phi_{lj} \, \partial_l \bigg[
		\frac{\partial_k \tilde u_i }{ ( |(\nabla\tilde{\bm u})\bm\Phi|^2 + \epsilon^2 )^{1/2} }
	\bigg] + \Phi_{ki} \partial_k \tilde p \in L^2(K_+) \quad (i = 1, \dots, d).
\end{equation*}
From the tangential regularity $\nabla\nabla' \tilde{\bm u} \in \bm L^2(K_+ \cap W)$ we see that (note that $S := \sum_{j=1}^d |\Phi_{dj}|^2 > 0$):
\begin{equation} \label{eq: 5.2}
	- \bigg[ 1 + \frac{\tilde g}{(|(\nabla\tilde{\bm u})\bm\Phi|^2 + \epsilon^2)^{1/2}} \bigg] \partial_d^2 \tilde u_i
	- \tilde g \, \frac{\Phi_{dj}}{S} \, (\partial_k \tilde u_i \Phi_{kj}) \, \partial_d \bigg[\frac{1}{(|(\nabla\tilde{\bm u})\bm\Phi|^2 + \epsilon^2)^{1/2}}\bigg]
	+ \frac{\Phi_{di}}{S} \partial_d \tilde p \in L^2(K_+ \cap W)
\end{equation}
for $i = 1, \dots, d$.
Here notice that the divergence-free condition $\partial_k\tilde u_i \Phi_{ki} = 0$ and the tangential regularity imply 
\begin{equation*}
	\Phi_{di} \partial_d^2 \tilde u_i \in L^2(K_+ \cap W).
\end{equation*}
This together with \eref{eq: 5.2}, which is multiplied by $\Phi_{di}$ and added for $i=1, \dots, d$, deduces
\begin{equation} \label{eq: 5.4}
	\partial_d \tilde p
	\in \tilde g \, \frac{\Phi_{di} \Phi_{dj}}{S} \, (\partial_k \tilde u_i \Phi_{kj}) \, \partial_d \bigg[\frac{1}{(|(\nabla\tilde{\bm u})\bm\Phi|^2 + \epsilon^2)^{1/2}}\bigg] + L^2(K_+ \cap W).
\end{equation}

Multiplying \eref{eq: 5.2} by $-\partial_d^2 \tilde u_i$ and adding the resulting equations for $i = 1, \dots, d$, we obtain
\begin{equation} \label{eq: 5.3}
\begin{aligned}
	&\bigg[ 1 + \tilde g \frac{|(\nabla\tilde{\bm u})\bm\Phi|^2 + \epsilon^2}{(|(\nabla\tilde{\bm u})\bm\Phi|^2 + \epsilon^2)^{3/2}} \bigg] |\partial_d^2 \tilde{\bm u}|^2 \\
	&\hspace{1cm} + \tilde g \, \frac{\Phi_{dj}}{S} \, (\partial_k \tilde u_i \Phi_{kj}) \Big( 1 - \frac{|\Phi_{di}|^2}S \Big) \, \partial_d \bigg[\frac{1}{(|(\nabla\tilde{\bm u})\bm\Phi|^2 + \epsilon^2)^{1/2}}\bigg] \, \partial_d^2 \tilde u_i
	= \bm R^{(1)} \cdot \partial_d^2 \tilde{\bm u},
\end{aligned}
\end{equation}
for some $\bm R^{(1)} \in \bm L^2(K_+ \cap W)$.
We have
\begin{align}
	\partial_d \bigg[\frac{1}{(|(\nabla\tilde{\bm u})\bm\Phi|^2 + \epsilon^2)^{1/2}}\bigg]
	&= -\frac{((\nabla\tilde{\bm u})\bm\Phi) : \partial_d((\nabla\tilde{\bm u})\bm\Phi)}{(|(\nabla\tilde{\bm u})\bm\Phi|^2 + \epsilon^2)^{3/2}} \notag \\
	&= -\frac{((\nabla\tilde{\bm u})\bm\Phi) : (\nabla \partial_d\tilde{\bm u})\bm\Phi}{(|(\nabla\tilde{\bm u})\bm\Phi|^2 + \epsilon^2)^{3/2}}
		-\frac{((\nabla\tilde{\bm u})\bm\Phi) : \big[ (\nabla \tilde{\bm u}) \partial_d \bm\Phi \big] }{(|(\nabla\tilde{\bm u})\bm\Phi|^2 + \epsilon^2)^{3/2}}, \label{eq: derivative of singular diffusion}
\end{align}
in which the second term in the right-hand side is bounded by $C (|(\nabla\tilde{\bm u})\bm\Phi|^2 + \epsilon^2)^{-1/2}$.
Therefore, if we set $\bm\Phi_d := (\Phi_{d1}, \dots, \Phi_{dd})$, the second term in the left-hand side of \eref{eq: 5.3} is bounded by
\begin{align*}
	&\frac{\tilde g}{S} \left| 
		(\partial_k \tilde u_i \Phi_{kj}) \, \partial_d \bigg[\frac{1}{(|(\nabla\tilde{\bm u}) \bm\Phi|^2 + \epsilon^2)^{1/2}}\bigg] \, \partial_d^2 \tilde u_i \Phi_{dj}
		\right| \\
	\le \; & \frac{\tilde g}{S} \frac{ |((\nabla\tilde{\bm u})\bm\Phi) : (\nabla \partial_d\tilde{\bm u})\bm\Phi| \,
		|((\nabla\tilde{\bm u})\bm\Phi) : (\partial_d^2 \tilde{\bm u}) \bm\Phi_d| }
	{(|(\nabla\tilde{\bm u}) \bm\Phi|^2 + \epsilon^2)^{3/2}}
		+ \big| \bm R^{(2)} \cdot \partial_d^2 \tilde{\bm u} \big| \\
	\le \; & \frac{\tilde g}{S} \frac{
		|(\nabla\tilde{\bm u}) \bm\Phi|^2 \, |\bm\Phi_d^\top |^2 \, |\partial_d^2 \tilde{\bm u}|^2 }
	{(|(\nabla\tilde{\bm u})\bm\Phi|^2 + \epsilon^2)^{3/2}}
		+ |R^{(3)}| \, |\partial_d^2 \tilde{\bm u} \big|
	= \tilde g \frac{ |(\nabla\tilde{\bm u}) \bm\Phi|^2 }{(|(\nabla\tilde{\bm u}) \bm\Phi|^2 + \epsilon^2)^{3/2}} |\partial_d^2 \tilde{\bm u}|^2
		+ |R^{(3)}| \, |\partial_d^2 \tilde{\bm u} \big|
\end{align*}
for some $\bm R^{(2)} \in \bm L^2(K_+ \cap W)$ and $R^{(3)} \in L^2(K_+)$.
Here we have used the fact that
\begin{equation*}
	\Big( \sum_{k=1}^{d-1} \partial_k (\partial_d \tilde u_i) \Phi_{kj} \Big) (|(\nabla\tilde{\bm u})\bm\Phi|^2 + \epsilon^2)^{-1/2}
		\in L^2(K_+ \cap W) \qquad (i, j = 1, \dots, d),
\end{equation*}
which is a consequence of the tangential regularity, in the third line.
Then it follows from \eref{eq: 5.3} that
\begin{equation*}
	\bigg[ 1 + \tilde g \frac{\epsilon^2}{(|(\nabla\tilde{\bm u}) \bm\Phi|^2 + \epsilon^2)^{3/2}} \bigg] |\partial_d^2 \tilde{\bm u}|^2 \le |R^{(4)}| + \frac12 |\partial_d^2 \tilde{\bm u}|^2 \quad\text{a.e. in}\quad K_+ \cap W,
\end{equation*}
where $R^{(4)} \in L^1(K_+)$, which implies $\partial_d^2 \tilde{\bm u} \in \bm L^2(K_+ \cap W)$.
By \eref{eq: 5.4} one also has $\partial_d \tilde p \in L^2(K_+ \cap W)$.
Combined with the tangential regularity, we have
\begin{equation*}
	\tilde{\bm u} \in \bm H^2(K_+ \cap W), \quad \tilde p \in H^1(K_+ \cap W),
\end{equation*}
and hence $\bm u \in \bm H^2(\Omega \cap W_{\bm x_0})$, $p \in H^1(\Omega \cap W_{\bm x_0})$.
This combined with the interior regularity and a standard covering argument yields
\begin{equation*}
	\bm u \in \bm H^2(\Omega), \quad p \in H^1(\Omega).
\end{equation*}

\subsubsection{Satisfaction of the slip boundary condition and $H^3$-approximation for $\bm u$}
Since $\bm u \in \bm H^2(\Omega)$ and $p \in H^1(\Omega)$, the traces of $\nabla\bm u$ and $p$ are well defined in $\bm H^{1/2}(\Gamma)$ and $H^{1/2}(\Gamma)$.
Let us confirm that $\tilde{\bm u}$ satisfies the perfect slip boundary condition.
\begin{lem} \label{lem: first-order trace}
	It holds that $(\nabla \bm u \, \bm n)_\tau = \bm 0$ a.e.\ on $\Gamma$.
	This implies, in the $\bm y$-coordinate, that
	\begin{equation} \label{eq: partial_d underline tilde u}
		\partial_d \underline{\tilde{\bm u}}' = 0, \quad 
		|\partial_d \underline{\tilde u}_d| \le C(|\nabla' \tilde{\bm u}'| + |\tilde{\bm u}'|)
		\quad\text{a.e.\ on }\; K_0,
	\end{equation}
	where we recall $\nabla' = (\partial_1, \dots, \partial_{d-1})^\top$ and $\underline{\tilde{\bm u}}' = (\underline{\tilde u}_1, \dots, \underline{\tilde u}_{d-1})^\top$.
\end{lem}
\begin{proof}
	Let $\bm\eta \in \bm H^{1/2}(\Gamma), \; \bm\eta \cdot \bm n = 0$ be arbitrary.
	It follows from integration by parts and \eref{eq: VEeps} that
	\begin{align*}
		-\bigg( \Big(1 + \frac{g}{\sqrt{|\nabla\bm u|^2 + \epsilon^2}} \Big) (\nabla\bm u \, \bm n)_\tau, \bm \eta \bigg)_\Gamma
		= (\bm f, \bm v) - (\nabla\bm u, \nabla\bm v) - \Big( g \frac{\nabla\bm u}{ \sqrt{|\nabla \bm u|^2 + \epsilon^2} }, \nabla \bm v \Big) + (p, \operatorname{div} \bm v) = 0,
	\end{align*}
	where $\bm v \in \bm V$ is any $\bm H^1(\Omega)$-extension of $\bm\eta$.
	Then $\big(1 + g/\sqrt{|\nabla \bm u|^2 + \epsilon^2} \big) (\nabla \bm u \, \bm n)_\tau = \bm 0$, and thus $(\nabla \bm u \, \bm n)_\tau = \bm 0$ a.e.\ on $\Gamma$.
	The first relation in \eref{eq: partial_d underline tilde u} is already obtained in \lref{lem: properties of new frame}(iii).
	To show the second one, observe that the divergence-free condition $\partial_k (\mathring\Psi_{ia} \underline{\tilde u}_a) \Phi_{ki} = 0$ yields
	\begin{equation} \label{eq: partial_d underline tilde u_d}
		\mathring\Psi_{id} (\partial_d \underline{\tilde u}_d) \Phi_{di} = -\sum_{(k, a) \neq (d, d)} \partial_k (\mathring\Psi_{ia} \underline{\tilde u}_a) \Phi_{ki} \quad\text{in }\; \bm H^1(K_+),
	\end{equation}
	in which the summation convention is still applied to $i$.
	Taking the trace to $y_d = 0$ (hence $\mathring{\bm\Psi} = \bm\Psi$) and utilizing the relations
	$\Psi_{ia} \Phi_{ki} = \delta_{ka}$, $\partial_d \underline{\tilde{\bm u}}' = \bm 0$, $\underline{\tilde u}_d = 0$, $\underline{\tilde{\bm u}}' = \tilde{\bm u}'$, and \eref{eq: Psi and Phi}, we obtain
	\begin{equation} \label{eq2: partial_d underline tilde u_d}
	\begin{aligned}
		\partial_d \underline{\tilde u}_d &= -\sum_{k, a = 1}^{d-1} (\partial_k \Psi_{ia}) \Phi_{ki} \underline{\tilde u}_a - \sum_{k, a=1}^{d-1} \delta_{ka} \partial_k \underline{\tilde u}_a
			= -\sum_{k, a = 1}^{d-1} (\partial_k \Psi_{da}) \Phi_{kd} \underline{\tilde u}_a - \sum_{k = 1}^{d-1} \partial_k \tilde u_k \\
			&= - \nabla^{\prime \top} \tilde{\bm u}' - (\nabla' \rho)^\top (\nabla^{\prime 2} \rho) \tilde{\bm u}'
			\quad\text{a.e.\ on }\; K_0,
	\end{aligned}
	\end{equation}
	where $\nabla^{\prime 2} \rho$ is the Hessian matrix of $\rho$ in $\bm y'$.
	This proves the second relation of \eref{eq: partial_d underline tilde u}.
\end{proof}
\begin{rem}
	As shown above, the assumption $g = 0$ on $\Gamma$ is not necessary here.
\end{rem}

Let us proceed to estimate $\|\partial_d^2 \tilde{\bm u}\|_{\bm L^2(K_+ \cap W)}$.
It is tempting to multiply \eref{eq: strong form in y}$_1$ by $-\partial_d (\tilde\theta^2 \partial_d\tilde u_i)$ and integrate over $K_+$.
However, for the integration by parts to be valid, we need---at least formally---third order derivatives of $\tilde{\bm u}$.
To address this issue, let us introduce an approximation $\bm u^{(m)} \in \bm H^3(\Omega)$ by solving, for $m = 1, 2, \dots$,
\begin{equation} \label{eq: H3 approximation}
\begin{aligned}
	- \Delta \bm u^{(m)} + \nabla q^{(m)} &= \bm h^{(m)} &&\text{in }\; \Omega, \\
	\operatorname{div} \bm u^{(m)} &= 0 &&\text{in }\; \Omega, \\
	\bm u^{(m)} \cdot \bm n = 0, \quad (\nabla \bm u^{(m)} \bm n)_\tau &= 0 &&\text{on }\; \Gamma,
\end{aligned}
\end{equation}
where $\bm h^{(m)} \in \bm H^1(\Omega)$ is chosen in such a way that
\begin{equation*}
	\bm h^{(m)} \to \bm h:= -\Delta \bm u + \nabla q \quad\text{ strongly in $\bm L^2(\Omega)$ as $m \to \infty$}, \quad q := 0.
\end{equation*}
By virtue of the $H^2$-regularity result for the linear Stokes equations under the perfect slip boundary condition (which may be proved similarly as in \cite[Theorem 1.2]{BdV04}), we find that
\begin{equation} \label{eq: um converges to u}
	\|\bm u^{(m)} - \bm u\|_{\bm H^2(\Omega)} \le C \|\bm h^{(m)} - \bm h\|_{\bm L^2(\Omega)} \to 0 \quad (m \to \infty).
\end{equation}

The traces of the second-order derivatives of $\bm u^{(m)}$ are now well defined in $\bm H^{1/2}(\Gamma)$, to which we obtain the following estimate.
We emphasize that the perfect slip boundary condition plays an essential role here.
\begin{lem} \label{lem: second-order trace}
	Let $\bm u^{(m)} \in \bm H^3(\Omega) \cap \bm V_\sigma$ be as above. Then we have
	\begin{equation} \label{eq: normal derivatives are bounded by tangential ones}
		\nabla' \partial_d \underline{\tilde{\bm u}}^{(m) \prime} = \bm 0, \quad |\partial_d^2 \underline{\tilde u}^{(m)}| \le C (|\nabla' \tilde{\bm u}^{(m) \prime}| + |\tilde{\bm u}^{(m) \prime}|) \quad\text{a.e.\ on }\; K_0.
	\end{equation}
\end{lem}
\begin{proof}
	The equality $\nabla' \partial_d \underline{\tilde{\bm u}}^{(m) \prime} = \bm 0$ results from \eref{eq: partial_d underline tilde u} (with $\bm u$ replaced by $\bm u^{(m)}$, which will not be emphasized again) differentiated in tangential directions.
	Next differentiate \eref{eq: partial_d underline tilde u_d} in $y_d$ to get
	\begin{align*}
		\mathring\Psi_{id} (\partial_d^2 \underline{\tilde u}^{(m)}_d) \Phi_{di}
		&= -\mathring\Psi_{id} (\partial_d \underline{\tilde u}^{(m)}_d) \partial_d \Phi_{di}
			-\sum_{(k, a) \neq (d, d)} \mathring\Psi_{ia} (\partial_k \partial_d \underline{\tilde u}^{(m)}_a) \Phi_{ki}
			-\sum_{(k, a) \neq (d, d)} \partial_k \mathring\Psi_{ia} (\partial_d \underline{\tilde u}^{(m)}_a) \Phi_{ki} \\
		&\qquad -\sum_{(k, a) \neq (d, d)} \mathring\Psi_{ia} (\partial_k \underline{\tilde u}^{(m)}_a) \partial_d \Phi_{ki}
			- \sum_{(k, a) \neq (d, d)} (\partial_k \mathring\Psi_{ia}) \underline{\tilde u}^{(m)}_a (\partial_d \Phi_{ki}) \qquad 
			\quad \text{in} \quad \bm H^1(K_+).
	\end{align*}
	Taking the trace to $y_d = 0$ and utilizing the relations
	$\Psi_{ia} \Phi_{ki} = \delta_{ka}$, $\partial_d \underline{\tilde{\bm u}}^{(m) \prime} = \bm 0$, $\underline{\tilde u}_d^{(m)} = 0$, $\underline{\tilde{\bm u}}^{(m) \prime} = \tilde{\bm u}^{(m) \prime}$, and \eref{eq: Psi and Phi}, we deduce that
	\begin{equation} \label{eq: partial_d^2 underline tilde u_d^m}
		\partial_d^2 \underline{\tilde u}^{(m)}_d = -\Psi_{id} (\partial_d \underline{\tilde u}^{(m)}_d) \partial_d \Phi_{di}
			- \sum_{k=1}^{d-1} \partial_k \Psi_{id} (\partial_d \underline{\tilde u}^{(m)}_d) \Phi_{ki}
			- \sum_{k, a = 1}^{d-1} \Psi_{ia} (\partial_k \tilde u^{(m)}_a) \partial_d \Phi_{ki}
			- \sum_{(k, a) \neq (d, d)} (\partial_k \Psi_{ia}) \underline{\tilde u}^{(m)}_a (\partial_d \Phi_{ki}).
	\end{equation}
	This combined with \eref{eq: partial_d underline tilde u} implies the desired inequality in \eref{eq: normal derivatives are bounded by tangential ones}.
\end{proof}

\subsubsection{Estimates for $\nabla \partial_d \tilde{\bm u}$ uniform in $\epsilon$}
We now multiply \eref{eq: strong form in y}$_1$ by $-\partial_d (\theta^2 \partial_d \tilde u_i)$, integrate over $K_+$, take the summation for $i = 1, \dots, d$.
The resulting equation reads
\begin{equation} \label{eq: starting pt of uniform normal estimates}
	I_1 + I_2 + I_3 = I_4,
\end{equation}
where
\begin{alignat*}{2}
	I_1 &= \int_{K_+} \partial_l \big( \partial_k \tilde u_i \Phi_{kj} \Phi_{lj} \big) \, \partial_d (\theta^2 \partial_d \tilde u_i) \, d\bm y, & \qquad
	I_2 &= \int_{K_+} \partial_l \bigg[ \tilde g \frac{ \partial_k \tilde u_i \Phi_{kj} }{ ( |(\nabla\tilde{\bm u})\bm\Phi|^2 + \epsilon^2 )^{1/2} } \Phi_{lj} \bigg] \, \partial_d (\theta^2 \partial_d \tilde u_i) \, d\bm y, \\
	I_3 &= \int_{K_+} \partial_k (\tilde p \, \Phi_{ki}) \big[ - \partial_d (\theta^2 \partial_d\tilde u_i) \big] \, d\bm y, & \qquad 
	I_4 & = \int_{K_+} \tilde F_i \big[ - \partial_d (\theta^2 \partial_d\tilde u_i) \big] \, d\bm y.
\end{alignat*}
Let us show that the places of $\partial_l$ and $\partial_d$ in $I_1$ and $I_2$ above can be interchanged up to lower-order terms.
We remark that the assumption $g \in H^1_0(\Omega)$, in particular $g = 0$ on $\Gamma$, is essentially utilized here.
\begin{lem}
	If $g \in H^1_0(\Omega)$, we have
	\begin{align*}
		I_1 &= \int_{K_+} \partial_d \big( \partial_k \tilde u_i \Phi_{kj} \Phi_{lj} \big) \, \partial_l (\theta^2 \partial_d \tilde u_i) \, d\bm y + L_1 =: I_{11} + L_1, \qquad
		|L_1| \le C\|\tilde{\bm u}\|_{\bm H^1(K_0 \cap (-r/2, r/2)^d)}^2, \\[3mm]
		I_2 &= \int_{K_+} \partial_d \bigg[ \tilde g \frac{ \partial_k \tilde u_i \Phi_{kj} }{ ( |(\nabla\tilde{\bm u})\bm\Phi|^2 + \epsilon^2 )^{1/2} } \Phi_{lj} \bigg] \, \partial_l (\theta^2 \partial_d \tilde u_i) \, d\bm y.
	\end{align*}
\end{lem}
\begin{proof}
	Let $\bm u^{(m)}$ be as in \eref{eq: H3 approximation} and consider $I_1$.
	We find from integration by parts that
	\begin{align*}
		&\int_{K_+} \partial_l \big( \partial_k \tilde u_i \Phi_{kj} \Phi_{lj} \big) \, \partial_d (\theta^2 \partial_d \tilde u_i^{(m)}) \, d\bm y \\
			= &\int_{K_+} \partial_d \big( \partial_k \tilde u_i \Phi_{kj} \Phi_{lj} \big) \, \partial_l (\theta^2 \partial_d \tilde u_i^{(m)}) \, d\bm y \\
			&\qquad + \int_{K_0} \underline n_l ( \partial_k \tilde u_i \Phi_{kj} \Phi_{lj} ) \, \partial_d (\theta^2 \partial_d \tilde u_i^{(m)}) \, d\bm y'
			- \int_{K_0} \underline n_d ( \partial_k \tilde u_i \Phi_{kj} \Phi_{lj} ) \, \partial_l (\theta^2 \partial_d \tilde u_i^{(m)}) \, d\bm y'
			=: I_{11}^{(m)} + I_{12}^{(m)} + I_{13}^{(m)},
	\end{align*}
	where we recall $\underline{\bm n} = (0, \dots, 0, -1)^\top$.
	Since $I_{11}^{(m)}$ converges to $I_{11}$ as $m \to \infty$ (recall \eref{eq: um converges to u}), it suffices to show that $I_{12}^{(m)} + I_{13}^{(m)}$ gives rise to lower-order terms.
	We see that $I_{12}^{(m)}$ is reduced to the term with $l = d$, which is canceled with $I_{13}^{(m)}$ with $l = d$.
	Therefore, it remains to consider each term of $I_{13}^{(m)}$ for a fixed $l < d$.
	Using the new frame we have
	\begin{equation*}
		I_{13}^{(m)} = \int_{K_0} \partial_k (\mathring\Psi_{ia} \underline{\tilde u}_a) \Phi_{kj} \, \partial_l \big( \theta^2 \partial_d (\mathring\Psi_{ib} \underline{\tilde u}_b^{(m)}) \big) \Phi_{lj} \, d\bm y'.
	\end{equation*}
	\begin{itemize}
		\item The terms for $k = d$ vanish because $\Phi_{dj} \Phi_{lj} = (|\nabla'\rho|^2+1)^{-1} \Psi_{dj} \Phi_{jl} = 0$ by \eref{eq: Psi and Phi}.
		\item The terms for $b < d$ vanish by \lref{lem: first-order trace} (applicable to $\tilde{\bm u}^{(m)}$ as well) and \lref{lem: second-order trace}.
		\item The terms with $k < d$ and $b = d$ are reduced, in view of \eref{eq: partial_d underline tilde u}, to
		\begin{equation*}
			\int_{K_0} \theta^2 \partial_k (\Psi_{ia} \underline{\tilde u}_a) \Phi_{kj} \, \Psi_{id} (\partial_l \partial_d \underline{\tilde u}_d^{(m)})  \Phi_{lj} \, d\bm y' + L_{131}^{(m)}, \quad |L_{131}^{(m)}| \le C\|\tilde{\bm u}\|_{\bm H^1(K_0 \cap (-r/2, r/2)^d)}^2.
		\end{equation*}
		If $\partial_k$ does not apply to $\Psi_{ia}$, then the relation $\Psi_{ia} \Psi_{id} = (|\nabla'\rho|^2 + 1) \delta_{ad}$ implies $a = d$, in which case $\underline{\tilde u}_a = 0$.
		If $\partial_k$ applies to $\Psi_{ia}$ and $a \neq d$, then $\partial_k\Psi_{ia}$ becomes non-zero only when $i = d$.
		Therefore, we are left with
		\begin{equation} \label{eq: this breaks down for singular diffusion}
			-\sum_{k, l, a = 1}^{d-1} \int_{K_0} \theta^2 (\partial_k \Psi_{da}) \Phi_{kj} \Phi_{lj} \cdot \underline{\tilde u}_a (\partial_l \partial_d \underline{\tilde u}_d^{(m)})  \, d\bm y',
		\end{equation}
		which can be bounded (use \eref{eq: partial_d underline tilde u} again) by $C\|\tilde{\bm u}\|_{\bm H^1(K_0 \cap (-r/2, r/2)^d)}^2$ after integration by parts for $\partial_l$.
	\end{itemize}
	
	Treatment of $I_2$ is similar and actually easier; in fact no boundary terms on $K_0$ appear in the integration by parts because $g = 0$ on $\Gamma$ by assumption.
\end{proof}
\begin{rem} \label{rem: H^1_0 is inevitable}
	(i) By a trace inequality, $|L_1|$ is further bounded by the square of the right-hand side of \eref{eq2: proof of tangential direction}.
	
	(ii) If $g \neq 0$ on the boundary, a similar argument as above applied to $I_2$ breaks down at a counter part of \eref{eq: this breaks down for singular diffusion}, where $\underline{\tilde u}_a$ is replaced by $\tilde g \underline{\tilde u}_a / \sqrt{|\nabla(\mathring{\bm\Psi} \underline{\tilde{\bm u}}) \bm\Phi|^2 + \epsilon^2}$ in the integrand.
	Note also that this term disappears if the boundary is flat, i.e., $\rho$ is a linear function.
\end{rem}

We split $I_{11}$ into a principal part and lower-order terms $L_{11}$ to derive
\begin{equation*}
	I_{11} = \int_{K_+} \theta^2 | \nabla (\partial_d \tilde{\bm u}) \bm\Phi|^2 \, d\bm y + L_{11}
		\ge C_* \|\theta \nabla \partial_d \tilde{\bm u}\|_{\bm L^2(K_+)}^2 + L_{11},
\end{equation*}
where we distinguish $C_*$ from other generic constants for convenience and
\begin{equation*}
	L_{11} \le C \|\nabla \tilde{\bm u}\|_{\bm L^2(K_+)} \big( \|\nabla \tilde{\bm u}\|_{\bm L^2(K_+)} + \|\theta \nabla \partial_d \tilde{\bm u}\|_{\bm L^2(K_+)} \big).
\end{equation*}

Next observe that $I_2$ is rewritten as
\begin{align*}
	I_2 &= \int_{K_+} \theta^2 \tilde g \, \partial_d \bigg[ \frac{ \partial_k \tilde u_i }{ ( |(\nabla\tilde{\bm u})\bm\Phi|^2 + \epsilon^2 )^{1/2} } \bigg] \Phi_{kj} \,  (\partial_l \partial_d \tilde u_i) \Phi_{lj} \, d\bm y
		\! + \! \int_{K_+} \tilde g \partial_d \bigg[ \frac{ \partial_k \tilde u_i }{ ( |(\nabla\tilde{\bm u})\bm\Phi|^2 + \epsilon^2 )^{1/2} } \bigg] \Phi_{kj} \, \partial_l (\theta^2) \partial_d \tilde u_i \Phi_{lj} \, d\bm y + L_2 \\
		&=: I_{21} + I_{22} + L_2,
\end{align*}
where the lower-order terms $L_2$ can be bounded by
\begin{align*}
	|L_2| &\le C \|\tilde g\|_{H^1(K_+)} \bigg\| \frac{\nabla \tilde{\bm u}}{( |(\nabla\tilde{\bm u}) \bm\Phi|^2 + \epsilon^2 )^{1/2}} \bigg\|_{\bm L^\infty(K_+)} ( \|\nabla \tilde{\bm u}\|_{\bm L^2(K_+)} + \|\theta \nabla \partial_d \tilde{\bm u}\|_{\bm L^2(K_+)} ) \\
		&\le C \|\tilde g\|_{H^1(K_+)} ( \|\nabla \tilde{\bm u}\|_{\bm L^2(K_+)} + \|\theta \nabla \partial_d \tilde{\bm u}\|_{\bm L^2(K_+)} ).
\end{align*}
Since (cf.\ \eref{eq: derivative of singular diffusion})
\begin{align*}
	\partial_d \bigg[ \frac{ \partial_k \tilde u_i }{ ( |(\nabla\tilde{\bm u})\bm\Phi|^2 + \epsilon^2 )^{1/2} } \bigg] \Phi_{kj}
	&= \frac{ (\partial_k \partial_d \tilde u_i) \Phi_{kj} }{ ( |(\nabla\tilde{\bm u})\bm\Phi|^2 + \epsilon^2 )^{1/2} }
		- \partial_k \tilde u_i \Phi_{kj} \frac{ ((\nabla\tilde{\bm u})\bm\Phi) : ((\nabla \partial_d\tilde{\bm u}) \bm\Phi) }{ ( |(\nabla\tilde{\bm u}) \bm\Phi|^2 + \epsilon^2 )^{3/2} } \\
		&\quad - \underbrace{ \partial_k \tilde u_i \Phi_{kj} \frac{ ((\nabla\tilde{\bm u})\bm\Phi) : \big[ (\nabla \tilde{\bm u}) \partial_d \bm\Phi \big] }{ ( |(\nabla\tilde{\bm u})\bm\Phi|^2 + \epsilon^2 )^{3/2} } }_{ |\cdot| \le C } \qquad (i, j = 1, \dots, d),
\end{align*}
we obtain, by the Cauchy--Schwarz inequality,
\begin{equation*}
	I_{21} = \int_{K_+} \theta^2 \tilde g
		\frac{ |\nabla (\partial_d \tilde{\bm u}) \bm\Phi|^2 \, 
		(|(\nabla\tilde{\bm u})\bm\Phi|^2 + \epsilon^2)
		- \big| ((\nabla\tilde{\bm u})\bm\Phi) : (\nabla (\partial_d\tilde{\bm u}) \bm\Phi) \big|^2 }
		{ ( |(\nabla\tilde{\bm u})\bm\Phi|^2 + \epsilon^2 )^{3/2} } d\bm y + L_{21}
	\ge 0 + L_{21},
\end{equation*}
where $|L_{21}| \le C \|\tilde g\|_{L^2(K_+)} \|\theta \nabla \partial_d \tilde{\bm u}\|_{\bm L^2(K_+)}$.
Similarly,
\begin{align*}
	I_{22} &= \int_{K_+} \tilde g \frac{ (\nabla \partial_d\tilde{\bm u})\bm\Phi : \big( \bm\Phi^\top\nabla(\theta^2) \otimes \partial_d\tilde{\bm u} \big) }{ (|(\nabla\tilde{\bm u})\bm\Phi|^2 + \epsilon^2 )^{1/2} } d\bm y \\
		&\hspace{2cm} - \int_{K_+} \tilde g \frac{ \big[ ((\nabla\tilde{\bm u})\bm\Phi) : ( \bm\Phi^\top\nabla(\theta^2) \otimes \partial_d\tilde{\bm u} ) \big] \, \big[ ((\nabla\tilde{\bm u})\bm\Phi) : ( (\nabla \partial_d\tilde{\bm u}) \bm\Phi) \big] }
		{ (|(\nabla\tilde{\bm u})\bm\Phi|^2 + \epsilon^2 )^{3/2} } d\bm y + L_{22},
\end{align*}
where $|L_{22}| \le C \|\tilde g\|_{L^2(K_+)} \|\nabla \tilde{\bm u}\|_{\bm L^2(K_+)}$.
Consequently,
\begin{equation*}
	|I_2| \le C \|\tilde g\|_{H^1(K_+)} ( \|\nabla \tilde{\bm u}\|_{\bm L^2(K_+)} + \|\theta \nabla \partial_d \tilde{\bm u}\|_{\bm L^2(K_+)} ).
\end{equation*}

We now focus on $I_3$.
It is represented as
\begin{equation*}
	I_3 = \lim_{m \to \infty} I_3^{(m)}, \qquad I_3^{(m)} := \int_{K_+} \partial_k (\tilde p \, \Phi_{ki}) \big[ - \partial_d (\theta^2 \partial_d\tilde u_i^{(m)}) \big] \, d\bm y.
\end{equation*}
Let us admit for the moment that the following key estimate holds:
\begin{equation} \label{eq: key estimate for pressure term}
\begin{aligned}
	|I_3^{(m)}| &\le C \big( \|\tilde{\bm f}\|_{\bm L^2(K_+)} + \|\tilde g\|_{H^1(K_+)} + \|\tilde{\bm u}\|_{\bm H^1(K_+)} + \|\tilde p\|_{L^2(K_+)} ) \\
		&\hspace{2cm} \times (\|\tilde{\bm u}^{(m)}\|_{\bm H^1(K_+)} + \|\theta \nabla \partial_d \tilde{\bm u}^{(m)}\|_{\bm L^2(K_+)} + \|\nabla' \tilde{\bm u}^{(m)}\|_{\bm H^1(K_+ \cap (-r/2, r/2)^d)}),
\end{aligned}
\end{equation}
which will be shown in Subsection \ref{sec: pressure on K0}.
Taking the limit $m \to \infty$, we find from \eref{eq2: proof of tangential direction} and the arithmetic-geometric mean inequality that
\begin{equation*}
	|I_3| \le C ( \|\tilde{\bm f}\|_{\bm L^2(K_+)}^2 + \|\tilde g\|_{H^1(K_+)}^2 + \|\tilde{\bm u}\|_{\bm H^1(K_+)}^2 + \|\tilde p\|_{L^2(K_+)}^2 )
		+ \frac{C_*}4 \|\theta \nabla\partial_d \tilde{\bm u}\|_{\bm L^2(K_+)}^2.
\end{equation*}
Finally, it is clear that $|I_4|$ can be bounded in a similar manner.

Combining the estimates for $I_1$ to $I_4$ with \eref{eq: starting pt of uniform normal estimates}, we conclude
\begin{equation*}
	\|\theta \nabla \partial_d \tilde{\bm u}\|_{\bm L^2(K_+)} \le C ( \|\tilde{\bm f}\|_{\bm L^2(K_+)} + \|\tilde g\|_{H^1(K_+)} + \|\tilde{\bm u}\|_{\bm H^1(K_+)} + \|\tilde p\|_{L^2(K_+)} ),
\end{equation*}
which completes the proof of \lref{lem: normal direction} except \eref{eq: key estimate for pressure term}.

\subsubsection{Proof of \eref{eq: key estimate for pressure term}} \label{sec: pressure on K0}
Integration by parts gives
\begin{align*}
	I_3^{(m)} &= \int_{K_+} \tilde p \Phi_{ki} \, \partial_k \partial_d (\theta^2 \partial_d\tilde u_i^{(m)}) \, d\bm y
		- \int_{K_0} \underline n_k \tilde p \, \Phi_{ki} \, \partial_d (\theta^2 \partial_d\tilde u_i^{(m)}) \, d\bm y'
		=: I_{31}^{(m)} + I_{32}^{(m)}.
\end{align*}
If we divide $I_{31}^{(m)}$ into a principal part and other lower-order terms represented by $L_{31}^{(m)}$, we get
\begin{equation*}
	I_{31}^{(m)} = \int_{K_+} \theta^2 \tilde p \, \partial_d^2 (\underbrace{ \partial_k \tilde u_i^{(m)} \Phi_{ki} }_{= 0}) \, d\bm y + L_{31}^{(m)}, \qquad
		|L_{31}^{(m)}| \le C \|\tilde p\|_{L^2(K_+)} \big( \|\nabla \tilde{\bm u}^{(m)}\|_{\bm L^2(K_+)} + \|\theta \nabla \partial_d \tilde{\bm u}^{(m)}\|_{\bm L^2(K_+)} \big),
\end{equation*}
which admits a bound as claimed in \eref{eq: key estimate for pressure term}.
\begin{rem}
	We see that $\nabla^2 \bm\Phi$, i.e., fourth-order derivatives of $\rho$ appear in $L_{31}^{(m)}$, which accounts for why we need $C^{3,1}$-regularity of $\Omega$.
\end{rem}

Thus it remains to estimate $I_{32}^{(m)}$, which is represented in the transformed frame as (note that $k$ must equal $d$, $\partial_d \mathring{\bm\Psi} = \bm O$, and $\mathring{\bm\Psi} = \bm\Psi$ on $K_0$)
\begin{equation*}
	I_{32}^{(m)} = \int_{K_0} \tilde p \underbrace{\Phi_{di} \, \Psi_{ib}}_{=\delta_{db}} \partial_d (\theta^2 \, \partial_d \underline{\tilde u}_b^{(m)}) \, d\bm y'
		= \int_{K_0} \tilde p \, \partial_d (\theta^2 \, \partial_d \underline{\tilde u}_d^{(m)}) \, d\bm y'.
\end{equation*}
To address this term, let us estimate the pressure trace $\tilde p|_{K_0}$ in $H^{-1/2}(K_0) + L^2(K_0)$ as follows.
\begin{lem} \label{lem: trace of p}
	For arbitrary $\eta \in H^{1/2}(K_0)$ with $\operatorname{supp}\eta \subset K_0 \cap (-r/2, -r/2)^d$ we have
	\begin{align*}
		\left| \int_{K_0} \tilde p \eta \, d\bm y' \right| &\le C \big( \|\tilde{\bm F}\|_{\bm L^2(K_+)} + \|\tilde g\|_{L^2(K_+)} + \|\nabla \tilde{\bm u}\|_{\bm L^2(K_+)} + \|\tilde p\|_{L^2(K_+)}) \|\eta\|_{H^{1/2}(K_0)} \\
		&\hspace{2cm} + C(\|\tilde{\bm u}\|_{\bm H^1(K_0 \cap (-r/2, r/2)^d)} + \|\tilde g\|_{L^2(K_0)} \big) \|\eta\|_{L^2(K_0)}.
	\end{align*}
\end{lem}
\begin{proof}
	There exists some $\tilde{\bm v} \in \bm H^1(K_+)$ satisfying $\tilde{\bm v} = \eta \bm N/\sqrt{|\nabla'\rho|^2 + 1}$ (hence $\underline{\tilde{\bm v}} = \eta \underline{\bm n}$ by \lref{lem: properties of new frame}(i)) on $K_0$, $\tilde{\bm v} = \bm 0$ on $\partial K_+ \setminus K_0$, and $\|\tilde{\bm v}\|_{\bm H^1(K_+)} \le C \|\eta\|_{H^{1/2}(K_0)}$.
	We multiply \eref{eq: strong form in y}$_1$ by $\tilde v_i$, integrate over $K_+$, perform integration by parts, and add for $i = 1, \dots, d$ to have
	\begin{align*}
		&\int_{K_+} \Big[
			(\partial_k \tilde u_i \Phi_{kj}) (\partial_l \tilde v_i \Phi_{lj})
			+ \tilde g \frac{\partial_k \tilde u_i \Phi_{kj} \; \partial_l \tilde v_i \Phi_{lj}}{(|(\nabla\tilde{\bm u})\bm\Phi|^2 + \epsilon^2)^{1/2}}
			- \tilde p \, \partial_k \tilde v_i \Phi_{ki} \Big] \, d\bm y \\
		&\qquad - \int_{K_0} \underline n_l \partial_k \tilde u_i \Phi_{kj} \, \tilde v_i \Phi_{lj} \, d\bm y'
			- \int_{K_0} \underline n_l \tilde g \frac{ \partial_k \tilde u_i \Phi_{kj} }{ (|(\nabla\tilde{\bm u})\bm\Phi|^2 + \epsilon^2)^{1/2} } \tilde v_i \Phi_{lj} \, d\bm y'
			+ \int_{K_0} \underline n_k \tilde p \Phi_{ki} \tilde v_i \, d\bm y'
			= \int_{K_+} \tilde F_i \tilde v_i \, d\bm y.
	\end{align*}
	
	Let us represent the boundary integrals in the transformed frame, in which we notice that $l$ and $k$ in the diffusion terms must equal $d$ and that $k = d$ in the pressure term.
	Then, since $\partial_d \underline{\tilde u}_a = 0$ for $a \neq d$, $\underline{\tilde v}_b = 0$ for $b \neq d$, and $\underline{\tilde v}_d = -\eta$ on $K_0$, it follows from \eref{eq: Psi and Phi} that
	\begin{equation*}
	\begin{aligned}
		&\underline n_l \partial_k \tilde u_i \Phi_{kj} \, \tilde v_i \Phi_{lj} = -\Phi_{dj}\Phi_{dj} \, \Psi_{ia} (\partial_d \underline{\tilde u}_a) \, \Psi_{ib}  \underline{\tilde v}_b
			= -\frac1{|\nabla'\rho|^2 + 1} \Psi_{id}\Psi_{id} (\partial_d \underline{\tilde u}_d) \underline{\tilde v}_d
			= (\partial_d \underline{\tilde u}_d) \eta, \\
		& \underline n_k  \tilde p \Phi_{ki} \tilde v_i = - \tilde p \underbrace{\Phi_{di} \Psi_{ib}}_{= \delta_{db}} \underline{\tilde v}_b = \tilde p \eta,
	\end{aligned}
	\qquad \text{on} \quad K_0.
	\end{equation*}
	Consequently,
	\begin{align*}
		\int_{K_0} \tilde p \eta \, d\bm y'
		&= - \int_{K_+} \Big[
			(\partial_k \tilde u_i \Phi_{kj}) (\partial_l \tilde v_i \Phi_{lj})
			+ \tilde g \frac{\partial_k \tilde u_i \Phi_{kj} \; \partial_l \tilde v_i \Phi_{lj}}{(|(\nabla\tilde{\bm u})\bm\Phi|^2 + \epsilon^2)^{1/2}}
			\tilde p \, \partial_k \tilde v_i \Phi_{ki} - \tilde F_i \tilde v_i \Big] \, d\bm y \\
		&\qquad + \int_{K_0} \partial_d \underline{\tilde u}_d \, \eta \, d\bm y'
			+ \int_{K_0} \tilde g \frac{ \partial_d \tilde u_i \Phi_{dj} \Phi_{dj} }{ (|(\nabla\tilde{\bm u})\bm\Phi|^2 + \epsilon^2)^{1/2} } \Psi_{id} \eta \, d\bm y'.
	\end{align*}
	By \eref{eq: partial_d underline tilde u} we obtain
	\begin{align*}
		\left| \int_{K_0} \tilde p \eta \, d\bm y' \right|
		&\le C \big( \|\tilde{\bm F}\|_{\bm L^2(K_+)} + \|\tilde g\|_{L^2(K_+)} + \|\nabla \tilde{\bm u}\|_{\bm L^2(K_+)} + \|\tilde p\|_{L^2(K_+)}) \|\tilde{\bm v}\|_{\bm H^1(K_+)} \\
		&\hspace{2cm} + C(\|\tilde{\bm u}\|_{\bm H^1(K_0 \cap (-r/2, r/2)^d)} + \|\tilde g\|_{L^2(K_0)} \big) \|\eta\|_{L^2(K_0)},
	\end{align*}
	which proves the lemma.
\end{proof}
\begin{rem}
	As shown above, the assumption $g = 0$ on $\Gamma$ is not necessary here.
\end{rem}

Setting $\eta := \partial_d (\theta^2 \, \partial_d\tilde u_d^{(m)}) \in H^1(K_+)$, we see from \eref{eq2: partial_d underline tilde u_d}, \eref{eq: partial_d^2 underline tilde u_d^m}, and a trace inequality that
\begin{align*}
	\|\eta\|_{H^{1/2}(K_0)} &\le C (\|\nabla' \tilde{\bm u}^{(m)}\|_{\bm H^{1/2}(K_0 \cap (-r/2, r/2)^d)} + \|\tilde{\bm u}^{(m)}\|_{\bm H^{1/2}(K_0)}) \\
		&\le C (\|\nabla' \tilde{\bm u}^{(m)}\|_{\bm H^1(K_+ \cap (-r/2, r/2)^d)} + \|\tilde{\bm u}^{(m)}\|_{\bm H^1(K_+)}).
\end{align*}
Applying \lref{lem: trace of p} to this $\eta$ and using \eref{eq: tilde F}, we conclude that $I_{32}^{(m)}$ also admits a bound as claimed in \eref{eq: key estimate for pressure term}.
This completes the proof of \lref{lem: normal direction} and thus that of \pref{main prop}.

\section{$H^2$-regularity for the non-stationary problem} \label{sec: non-stationary}
In this section, we focus on the non-stationary Bingham--Navier--Stokes problem \eref{eeq: strong form}--\eref{eeq: IC}, retrieving the notation $\nu$ to represent the positive viscosity constant.
Similarly to \pref{prop: VIsigma}(i), one can interpret \eqref{eeq: strong form}--\eqref{eeq: slipBC} as a variational inequality (cf.\ \cite[Section VI.3.2]{DuLi1976}) as follows:
\begin{equation} \label{wf}
	( \bm u'(t), \bm v -\bm u(t) )
	 - \nu ( \Delta \bm u(t),\bm v -\bm u(t))
	 + b( \bm u(t), \bm u(t), \bm v - \bm u(t))
	{} +  ( g(t), |\nabla \bm v| - |\nabla \bm u(t)| ) 
	\ge 
	( \bm f(t), \bm v - \bm u(t) ) \quad \forall \bm v \in \bm V_\sigma
\end{equation}
and for a.a.\ $t \in (0,T)$ with the initial condition $\bm u(0)=\bm u_0$ in $\bm H_\sigma$.
Here, $\bm u'$ means the time derivative $\partial \bm u/\partial t$, and the trilinear form $b$ is given by
\begin{equation*}
	b(\bm u, \bm v, \bm w)= \int_\Omega u_i \frac{\partial v_j}{\partial x_i} w_j dx 
	\qquad \forall \bm u, \bm v, \bm w \in \bm V_\sigma. 
\end{equation*}
We also define $B: \bm V_\sigma \to \bm V_\sigma'$ by $B \bm u:=P[(\bm u\cdot \nabla)\bm u]$, where  
$P:\bm L^2(\Omega) \to \bm H_\sigma$ is the Helmholtz--Leray projection, which can be naturally extended to $\bm V_{\sigma}' \to \bm V_\sigma'$. 
Then we have $\langle B \bm u , \bm v \rangle = b(\bm u, \bm u, \bm v)$ (recall $\langle \cdot, \cdot \rangle$ denotes the duality pairing between $\bm V_\sigma'$ and $\bm V_\sigma$) and
\begin{equation}
	\|  B \bm v-B \bm w \|_{\bm V_\sigma'} 
	\le C_{\rm L} 
	\| \bm v - \bm w  \|_{\bm V_\sigma} 
	\{\| \bm v \|_{\bm V_\sigma}+ \| \bm w \|_{\bm V_\sigma} \}
	\label{contib}
\end{equation}
for some constant $C_{\rm L} = C(\Omega) >0$, which implies the continuity of $B$ (cf.\ \cite[pp.\ 251--253]{Bar10}).

Under these settings we obtain the following extension of \cite[Theorem~5.9]{Bar10} to the Bingham--Navier--Stokes problem.
\begin{thm} \label{thm: evolution}
Let $\Omega$ be of $C^{3,1}$-class, $\bm f \in W^{1,1}(0, T; \bm L^2(\Omega))$, and $g \in H^1(0, T;  L^2(\Omega)) \cap L^\infty(0, T; H^1_0(\Omega))$, $g \ge 0$. 
Assume that there exists $\bm f^* \in \bm L^2(\Omega)$ such that $\bm u_0 \in \bm V_\sigma \cap \bm H^2(\Omega)$ satisfies
\begin{gather}
	(\nabla \bm u_0 \bm n)_\tau = \bm 0 \quad \text{a.e.\ on } \; \Gamma, \label{initialc1} \\
	(-\nu \Delta \bm u_0 + B \bm u_0 - \bm f(0), \bm v - \bm u_0) + ( g(0), |\nabla \bm v| - |\nabla \bm u_0| ) \ge (\bm f^*, \bm v - \bm u_0) \qquad \forall \bm v \in \bm V_\sigma. \label{initialc}
\end{gather}
Then, there exist 
$T_* \in (0,T]$ and a unique function $\bm u \in W^{1,\infty}(0,T_*;\bm H_\sigma) \cap H^1(0,T_*;\bm V_\sigma) \cap L^\infty (0,T_*;\bm H^2(\Omega))$
such that $\bm u$ satisfies $\bm u(0)=\bm u_0$ and \eref{wf} for a.a.\ $t \in (0,T_*)$. 
If $d=2$, $T_*$ can be taken as $T$.
\end{thm}
\begin{rem}
	Assumption \eref{initialc} is a kind of a compatibility condition for initial data and can be expressed as $\partial \varphi^0(\bm u_0) + B\bm u_0 - P\bm f(0) \cap \bm H_\sigma \neq \emptyset$ ($\varphi^t$ is defined in Subsection \ref{subsec: limit in N} below).
\end{rem}

We can obtain another strong solution with less regular $\bm f$, $g$, and $\bm u_0$ as follows.
Note that the compatibility condition \eref{initialc} is not necessary to assume in this case.
\begin{thm} \label{thm: evolution weak}
	Let $\Omega$ be of $C^{3,1}$-class, $\bm f \in L^2(0,T;\bm L^2(\Omega))$, $g \in W^{1,1}(0,T; L^2(\Omega)) \cap L^2(0,T; H^1_0(\Omega))$, $g \ge 0$, and $\bm u_0 \in \bm V_\sigma$. 
	Then, there exist $T_* \in (0,T]$ and a unique function $\bm u \in H^1(0,T_*;\bm H_\sigma) \cap C([0,T_*]; \bm V_\sigma) \cap L^2 (0,T_*;\bm H^2(\Omega))$ such that $\bm u$ satisfies $\bm u(0)=\bm u_0$ and \eqref{wf} for a.a.\ $t \in (0,T_*)$.
	If $d=2$, $T_*$ can be taken as $T$.
\end{thm}
\begin{rem} 
	(i) Since the embedding $H^1(0,T_*; \bm L^2(\Omega)) \cap L^2 (0,T_*; \bm H^2(\Omega)) \hookrightarrow C([0,T_*]; \bm H^1(\Omega))$ holds, it suffices to construct a solution $\bm u \in H^1(0,T_*;\bm H_\sigma) \cap L^2 (0,T_*;\bm H^2(\Omega))$.
	
	(ii) The well-posedness in this class may be discussed also from the abstract theory of evolution inclusion governed by the subdifferential with nonmonotone perturbation \cite{Ota82}.
	However, in order to apply this abstract theory we need stronger assumption on $g$ with respect to $t$, more precisely the Lipschitz continuity (see \cite[p.\ 273]{Ota82}). 
	In this sense, our assumption of Theorem~\ref{thm: evolution weak} with respect to $g$ is weaker than \cite{Ota82}.
\end{rem}

\subsection{Truncation of $B$}
The proof of the above theorems is based on the truncation method. 
To this end, for $M > 0$ we introduce a truncated operator $B_M: \bm V_\sigma \to \bm V_\sigma'$ by
\begin{equation*}
	B_M \bm v =
	\begin{cases}
	B \bm v & {\rm if~} \| \bm v \|_{\bm V_\sigma} \le M, \\
	\displaystyle \frac{M^2}{\| \bm v \|_{\bm V_\sigma}^2} B\bm v
	& {\rm if~} \| \bm v \|_{\bm V_\sigma} > M.\\
	\end{cases} 
\end{equation*}
The essential idea of this truncation goes back to \cite{Bar97, Bar10, BS01}, see also a similar method in \cite{CRK06, DD11}.
It follows from \cite[Chapter~5]{Bar10} that if $\bm v \in \bm V \cap \bm  H^2(\Omega)$ then $B_M \bm v$ makes sense in $\bm H_\sigma$ and
\begin{alignat}{2}
	\| B_M \bm v\|_{\bm H_\sigma}
	&\le C_{\rm L} M^{\frac{3}{2}} 
	\| \bm v \| _{\bm H^2(\Omega)}^{\frac{1}{2}},
	\label{bm}\\
	\| B_M \bm v \|_{\bm H_\sigma}
	&\le C_{\rm L} 
	\| \bm v \| _{\bm H_\sigma}^{\frac{1}{2}}
	\| \bm v \|_{\bm V_\sigma}
	\| \bm v \| _{\bm H^2(\Omega)}^{\frac{1}{2}}
	\quad &{\rm if ~} d = 2,
	\label{bm2}\\
	\| B_M \bm v \|_{\bm H_\sigma}
	&\le C_{\rm L} \| \bm v \|_{\bm V_\sigma}^{\frac{3}{2}} 
	\| \bm v \| _{\bm H^2(\Omega)}^{\frac{1}{2}}
	\quad &{\rm if ~} d = 3,
	\label{bm3}
\end{alignat}
for some constant $C_{\mathrm{L}} = C(\Omega) > 0$ independent of $M$.
Moreover, similarly to \eqref{contib}, for all $\bm v, \bm w \in \bm V_\sigma$ we have
\begin{equation}
	\| B_M \bm v-B_M \bm w \|_{\bm V_\sigma'} 
	\le 2C_{\rm L} \| \bm v - \bm w \|_{\bm V_\sigma} 
	\{ 
	\| \bm v \|_{\bm V_\sigma}+\| \bm w \|_{\bm V_\sigma} 
	\}.
	\label{contibm}
\end{equation}

Given a nonnegative function $g \in H^1_0(\Omega)$, define a proper, lower semi-continuous, and convex functional 
$j_g : \bm H_\sigma \to [0,+\infty]$ by 
\begin{equation*}
	j_g(\bm v) = \frac{\nu}{2} \int_\Omega |\nabla \bm v|^2 \, d\bm x + \int_\Omega g |\nabla \bm v| \, d\bm x \quad \text{if } \; \bm v \in D(j_g) := \bm V_\sigma, \qquad
	j_g(\bm v) =+\infty \quad \text{if } \; \bm v \in \bm H_\sigma \setminus \bm V_\sigma.
\end{equation*}
By the definition of a subdifferential, we have $\bm f \in \partial j_g(\bm u)$ in $\bm H_\sigma$ if and only if $\bm u \in \bm V_\sigma$ satisfies
\begin{equation}\label{sf}
		\nu (\nabla \bm u, \nabla (\bm v - \bm u) ) + (g, |\nabla \bm v| - |\nabla \bm u|) 
		\ge ( \bm f, \bm v - \bm u )\qquad \forall \bm v \in \bm V_\sigma.
\end{equation}
Then \cref{cor: main result} implies that $\bm u \in \{ \bm v \in \bm V_\sigma \cap \bm H^2(\Omega) \mid (\nabla \bm v \, \bm n)_\tau = \bm 0 \; \text{ a.e.\ on } \Gamma \} =: \bm W$ and that
\begin{equation} \label{K-estimate}
	\| \bm u \|_{\bm H^2(\Omega)} \le C_{\rm K} (\| \bm f\|_{\bm L^2(\Omega)}+ \|g\| _{H^1(\Omega)}),
\end{equation}
where $C_{\rm K} = C(\Omega, \nu)$ is independent of $\bm f$ and $g$.
In particular, we have $D(\partial j_g) \subset \bm W$.

Under these settings, according to \cite{Bar10} we define an operator $\Gamma_{g, M} :\bm H_\sigma \to 2^{\bm H_\sigma}$ by
\begin{equation*}
	\Gamma_{g, M} = \partial j_g + B_M, \quad D(\Gamma_{g, M}) = D(\partial j_g).
\end{equation*}

\begin{lem}\label{lem: maximal monotone}
	Let $g \in H^1_0(\Omega)$ and $M \in (0, \infty)$.
	Then there exists a constant $\alpha_M = C(\Omega, \nu, M) > 0$ independent of $g$ such that $\Gamma_{g, M} + \alpha I$ is maximal monotone 
	in $\bm H_\sigma$ for all $\alpha \ge \alpha_M$, where $I$ denotes the identity operator.
\end{lem}

\begin{proof}
We closely follow the arguments of \cite[Lemma 5.2]{Bar10}.
We remark that, for operators in a Hilbert space, the concepts of maximal monotonicity (where its dual is identified with itself) and $m$-accretivity coincide.

\textbf{Step 1.}
By \cite[pp.\ 254--255]{Bar10} there exists a positive constant $c_M = C(\Omega, \nu, M)$ such that
\begin{equation}
	| \langle B_M \bm w -B_M \bm z, \bm w -\bm z \rangle|
	\le \frac{\nu}{8} \|\bm w -\bm z \|_{\bm V_\sigma}^2 
	+c_M\| \bm w - \bm z \|_{\bm H_\sigma}^2 
	\qquad \forall \bm w, \bm z \in \bm V_\sigma.
	\label{eq: bound}
\end{equation}
Therefore, we see from \eqref{sf}, \eqref{eq: bound}, and \eqref{eq: coercivity} that for all $\bm w^* \in ((1/8)\partial j_g + B_M + \alpha I) \bm w$, $\bm z^* \in ((1/8)\partial j_g + B_M + \alpha I) \bm z$, and $\alpha \ge \alpha_M := c_M$, there exist $\bm f \in (1/8) \partial j_g (\bm w)$ and $\bm h \in (1/8) \partial j_g (\bm z)$ such that
\begin{align*}
	(\bm w^* -\bm z^*, \bm w-\bm z)
	& =
	(\bm f -\bm h, \bm w-\bm z)
	+
	 \langle B_M \bm w -B_M \bm z, \bm w - \bm z \rangle
	+ \alpha \| \bm w - \bm z\|_{\bm H_\sigma}^2 \\
	& \ge 
	 \frac{\nu}{8} (\nabla \bm w, \nabla (\bm w - \bm z)) + \frac18 (g, |\nabla \bm w| - |\nabla \bm z|) 
	+ \frac{\nu}{8} (\nabla \bm z, \nabla (\bm z - \bm w)) + \frac18 (g, |\nabla \bm z| - |\nabla \bm w|) \notag  \\
	& \qquad
	- \frac{\nu}{8} \| \bm w - \bm  z \|_{\bm V_\sigma}^2 
	+(\alpha- c_M) \| \bm w - \bm z \|_{\bm H_\sigma}^2 
	\\
	& = (\alpha- c_M) \| \bm w - \bm z \|_{\bm H_\sigma}^2 \ge 0.
\end{align*}
This implies the monotonicity of $(1/8)\partial j_g + B_M + \alpha I$, and hence that of $\Gamma_{g, M} + \alpha I$. 

\textbf{Step 2.}
Let us consider $(1/4) \partial_{*} j_g+ B_M: \bm V_\sigma \to 2 ^{\bm V_\sigma'}$ where $\partial_{*}$ stands for the subdifferential between $\bm V_\sigma$ and $\bm V_\sigma'$. 
Recall that $\partial j_g (\bm u) \subset \partial_{*} j_g(\bm u)$ in $\bm V_\sigma'$ and that if $\bm f \in \bm H_\sigma$ and $\bm u \in \bm V_\sigma$ satisfy $\bm f \in \partial_{*} j_g(\bm u)$ then $\partial j_g (\bm u) = \partial_{*} j_g(\bm u)$ in $\bm H_\sigma$.
Now we prove that $(1/4) \partial_{*} j_g + B_M + \alpha I : \bm V_\sigma \to 2 ^{\bm V'_\sigma}$ is maximal monotone.
For this we define
\begin{equation*}
	j_1 (\bm v) = \frac{\nu}{4} \int_\Omega |\nabla \bm v|^2 \, d\bm x + \int_\Omega g |\nabla \bm v| \, d\bm x, \quad
	j_2 (\bm v) = \frac{\nu}{4} \int_\Omega |\nabla \bm v|^2 \, d\bm x, \quad D(j_1) = D(j_2) = \bm V_\sigma.
\end{equation*}
Then $(1/4) \partial_{*} j_1 :\bm V_\sigma \to 2^{\bm V_\sigma'}$ is maximal monotone with $D((1/4) \partial_* j_1)=\bm V_\sigma$.
Moreover, $(1/4) \partial_{*} j_2 + B_M + \alpha I : \bm V_\sigma \to \bm V_\sigma'$ is single-valued, monotone (by a similar computation to Step 1), and demi-continuous.
Therefore, it follows from the Rockafellar theorem \cite[Corollary 2.6]{Bar10} that the sum
\begin{equation*}
	\frac14 \partial_{*} j_g + B_M + \alpha I = \frac14 \partial_{*} j_1 + \Big( \frac14 \partial_{*} j_2 + B_M + \alpha I \Big) : \bm V_\sigma \to 2^{\bm V_\sigma'}
\end{equation*}
is maximal monotone.
Furthermore, again by a similar computation to Step 1, it is coercive.
Hence we can apply the Minty theorem \cite[Corollary 2.2]{Bar10} to find that it is surjective, i.e., $R( (1/4) \partial_{*} j_g + B_M + \alpha I) = \bm V'_\sigma$.

\textbf{Step 3.}
Let us define $F_1 : \bm H_\sigma \to 2^{\bm H_\sigma}$ and $F_2 : \bm H_\sigma \to 2^{\bm H_\sigma}$ by
\begin{alignat*}{2}
	F_1 &= \frac34 \partial j_g, \qquad && D(F_1) = D(\partial j_g) \subset \bm W \subset \bm H^2(\Omega), \\
	F_2 &= \frac14 \partial_{*} j_g+ B_M + \alpha I, \qquad && D(F_2) = \Big\{ \bm v \in \bm V_\sigma \,\Big|\, \frac14 \partial_{*} j_g(\bm v) + B_M \bm v \subset \bm H_\sigma \Big\}.
\end{alignat*}
Note that $F_2$ is the restriction of $(1/4) \partial_{*} j_g + B_M + \alpha I$ to $\bm H_\sigma$.
It is surjective as a result of Step 2, and so is $F_2 + I$.
Therefore, again by the Minty theorem, it is maximal monotone (cf.\ \cite[Exemple 2.3.7]{Bre73}).

\textbf{Step 4.}
It remains to show the maximal monotonicity of $F_1 + F_2$ with the domain $D(\partial j_g)$.
Recall that the minimal section of $F_2$ is given by
\begin{equation*}
	F_2^\circ \bm v
	:=\left\{ \bm v^*_{\rm min} \in F_2 \bm v \,\middle|\,  \|\bm v^*_{\rm min} \|_{\bm H_\sigma} 
	=\min_{\bm v^* \in F_2 \bm v} \| \bm v^* \|_{\bm H_\sigma} \right\}.
\end{equation*}
Let $\bm u \in D(\partial j_g) \subset \bm W$ and $\bm f \in \partial_{*} j_g(\bm u)=\partial j_g(\bm u)$ be arbitrary.
Then, from \eqref{bm} and \eqref{K-estimate} we obtain
\begin{align*}
	\| F_2^\circ \bm u \|_{\bm H_\sigma} 
	& \le 
	\frac14 \|\bm f \|_{\bm H_\sigma} + 
	\|B_M \bm u \|_{\bm H_\sigma} 
	+ \alpha \| \bm u \|_{\bm H_\sigma}
	\le \frac14 \|\bm f \|_{\bm H_\sigma}
	+  C_{\rm L} M^{\frac{3}{2}} 
	\|\bm u \| _{\bm H^2(\Omega)}^{\frac{1}{2}}
	+ \alpha \| \bm u\|_{\bm H_\sigma} 
	\\
	& \le \frac14 \|\bm f \|_{\bm H_\sigma}
	+ C_{\rm L} M^{\frac{3}{2}} C_{\rm K}^{\frac{1}{2}} \|\bm f \| _{\bm H_\sigma}^{\frac{1}{2}}
	+ C_{\rm L} M^{\frac{3}{2}}  C_{\rm K}^{\frac{1}{2}} \|g\|_{H^1(\Omega)}^{\frac12}
	+ \alpha \| \bm u \|_{\bm H_\sigma} \\
	& \le \frac12 \|\bm f \|_{\bm H_\sigma} 
	+ C_{\rm L}^2 M^3  C_{\rm K}
	+  C_{\rm L} M^{\frac{3}{2}}  C_{\rm K}^{\frac{1}{2}} \|g\|_{H^1(\Omega)}^{\frac12} + \alpha \| \bm u \|_{\bm H_\sigma},
\end{align*}
which implies, with some constant $C = C(M, C_{\rm L}, C_{\rm K}, \alpha, g) > 0$,
\begin{equation*}
	\|F_2^\circ \bm u \|_{\bm H_\sigma} \le \frac23 \|F_1^\circ \bm u \|_{\bm H_\sigma} + C (\|\bm u\|_{\bm H_\sigma} +1).
\end{equation*}
Therefore, we conclude from a perturbation result for $m$-accretive operators (see \cite[Proposition~3.9]{Bar10} or \cite[Corollaire~2.6]{Bre73}) that $F_1 + F_2$ is maximal monotone in $\bm H_\sigma$, especially, $D(F_1 + F_2) = D(\partial j_g)$.
\end{proof}

\subsection{Discretization in time and a priori estimates}
To show Theorems \ref{thm: evolution} and \ref{thm: evolution weak}, we construct an approximate solution of \eref{wf} by the method of discretization in time.
For this purpose, fix $M > 0$, choose a division number $N \in \mathbb N$ of the time interval $[0, T]$ to be sufficiently large (at least $N/T \ge 4\alpha_M$), and define the time increment $h_N=T/N$.
We introduce discretizations of $\bm f \in L^2(0, T; \bm L^2(\Omega))$ and $g \in W^{1,1}(0, T; L^2(\Omega)) \cap L^2(0, T; H^1_0(\Omega))$ by
\begin{equation*}
	\bm f_k = \frac{1}{h_N} \int_{(k-1) h_N}^{kh_N} \bm f(s) \, ds \in \bm L^2(\Omega) \quad\text{and}\quad g_k = \frac{1}{h_N} \int_{(k-1) h_N}^{k h_N} g(s) \, ds \in H^1_0(\Omega) \qquad (k = 1, \dots, N),
\end{equation*}
together with $g_0 := g(0) \in L^2(\Omega)$.
We also use $\bm f_0 := \bm f(0) \in \bm L^2(\Omega)$, which makes sense in the setting of \tref{thm: evolution}.
Then we set $\varphi_k := j_{g_k}$, that is,
\begin{equation*}
	\varphi_k (\bm v) = \frac\nu2 \int_\Omega |\nabla \bm v|^2 \, d \bm x + \int_\Omega g_k |\nabla \bm v| \, d \bm x, \quad D(\varphi_k) = \bm V_\sigma \qquad (k = 1, \dots, N).
\end{equation*}

Recalling that $\bm u_0 \in \bm V_\sigma$ is given as an initial datum, we consider the following differential inclusions:
\begin{equation} \label{eq: discrete-in-time inclusion}
	\frac{\bm u_k-\bm u_{k-1}}{h_N} + \partial \varphi_k (\bm u_k)  + B_M \bm u_k \ni P \bm f_k \quad\text{in }\; \bm H_\sigma \qquad (k = 1, \dots, N),
\end{equation}
which is equivalent to
\begin{equation} \label{vf}
	\Big( \frac{\bm u_k-\bm u_{k-1}}{h_N}, \bm u_k-\bm v \Big) + \nu (\nabla \bm u_k, \nabla (\bm u_k-\bm v) ) + (B_M \bm u_k, \bm u_k-\bm v) + ( g_k, |\nabla \bm u_k| -  |\nabla \bm v|) \le (\bm f_k, \bm u_k-\bm v) \quad \forall \bm v \in \bm V_\sigma.
\end{equation}
Since $\partial \varphi_k + B_M + (1/h_N)I : \bm H_\sigma \to 2^{\bm H_\sigma}$ is maximal monotone by \lref{lem: maximal monotone}, there exists $\bm u_k \in D(\partial \varphi_k) \subset \bm W \subset \bm H^2(\Omega)$ satisfying \eref{eq: discrete-in-time inclusion} for all $k = 1, \dots, N$.

\begin{lem}\label{1stest} 
	Let $\{\bm u_k\}_{k=1}^N$ be defined as above.
	Then, for $m = 1,2, \ldots, N$ we have
	\begin{align*}
		\|\bm u_m \|_{\bm H_\sigma}^2 + \sum_{k=1}^m \|\bm u_k - \bm u_{k-1}\|_{\bm H_\sigma}^2 + \nu \sum_{k=1}^m \| \bm u_k \|_{\bm V_\sigma}^2 h_N + \sum_{k=1}^m ( g_k, |\nabla \bm u_k |) h_N \le C(\| \bm u_0 \|_{\bm H_\sigma}^2 + \|\bm f\|_{L^2(0,T;\bm L^2(\Omega))}^2),
	\end{align*}
	where the constant $C = C(\Omega, \nu)$ is independent of $m$, $N$, and $M$.
\end{lem}

\begin{proof}
Taking $\bm v = \bm 0$ in \eref{vf}, using the relations (recall \eqref{eq: coercivity} as well)
\begin{gather*}
	\left( \frac{\bm u_k- \bm u_{k-1}}{h_N}, 
	\bm u_k \right)
	= \frac{1}{2 h_N} ( \| \bm u_k \|_{\bm H_\sigma}^2 
	- \| \bm u_{k-1} \|_{\bm H_\sigma}^2
	+ \| \bm u_k - \bm u_{k-1}\|_{\bm H_\sigma}^2 ), \\ 
	(B_M \bm u_k,\bm u_k)=0, \quad 
	(\bm f_{k}, \bm u_k ) \le \frac{\nu}{2} \|\nabla \bm u_k \|_{\bm L^2(\Omega)}^2 + C \| \bm f_{k} \|_{\bm L^2(\Omega)}^2,
\end{gather*}
and adding the resulting inequalities for $k = 1, \dots, m$, we obtain the desired estimate.
\end{proof}

Next we obtain an $\ell^\infty(0, T; \bm H^1(\Omega))$-a priori estimate for $\{\bm u_k\}_{k=1}^N$.
To this end we rewrite \eref{eq: discrete-in-time inclusion} into the equation
\begin{equation} \label{dee}
	\frac{\bm u_k-\bm u_{k-1}}{h_N} + \bm u_k^* + B_M \bm u_k = P \bm f_k \quad \text{with some } \bm u_k^* \in \partial \varphi_k (\bm u_k) \qquad (k = 1, \dots, N).
\end{equation}

\begin{lem}\label{2ndest}
For $m = 1, \dots, N$ we have
\begin{equation} \label{eq: 2ndest}
\begin{aligned}
	& \| \nabla \bm u _m \|_{\bm L^2(\Omega)}^2 + ( g_m, |\nabla \bm u_m | ) + \sum_{k=1}^m \| \bm u^*_k \|_{\bm H_\sigma}^2 h_N + \sum_{k=1}^m \| \bm u _{k} \|_{\bm H^2(\Omega)}^2 h_N + \sum_{k=1}^m \left\| \frac{\bm u_k-\bm u_{k-1}}{h_N} \right\|_{\bm H_\sigma}^2 h_N \\
		&\qquad \le C_M (\|\nabla\bm u_0\|_{\bm L^2(\Omega)}^2 + \|\bm f\|_{L^2(0,T;\bm L^2(\Omega))}^2 + \|g\|_{L^2(0,T; H^1(\Omega))}^2 + \|g\|_{W^{1,1}(0,T;L^2(\Omega))}^2),
\end{aligned}
\end{equation}
where the constant $C_M = C(\Omega, \nu, T, M)$ is independent of $m$ and $N$.
\end{lem}

\begin{proof}
	By the definition of a subdifferential, $\varphi_k(\bm u_k) + (\bm u_{k-1} - \bm u_k, \bm u_k^*) \le \varphi_k(\bm u_{k-1})$.
	Substituting \eref{dee} gives
	\begin{equation*}
		\varphi_k(\bm u_k) + \|\bm u^*_k \|_{\bm H_\sigma}^2 h_N + (B_M\bm u_k - \bm f_k, \bm u_k^*)h_N  \le \varphi_k(\bm u_{k-1}),
	\end{equation*}
	so that, for $k = 1, \dots, N$,
	\begin{equation} \label{eq2: proof of 2ndest}
	\begin{aligned}
		&\frac\nu2 (\|\nabla \bm u_k\|_{L^2(\Omega)}^2 - \|\nabla \bm u_{k-1}\|_{L^2(\Omega)}^2) + (g_k, |\nabla \bm u_k|) - (g_{k-1}, |\nabla \bm u_{k-1}|) + \|\bm u^*_k \|_{\bm H_\sigma}^2 h_N \\
		\le \; &(\bm f_k - B_M\bm u_k, \bm u_k^*)h_N + (g_k - g_{k-1}, |\nabla \bm u_{k-1}|) \\
		\le \; &\frac14 \|\bm u^*_k \|_{\bm H_\sigma}^2 h_N + \|\bm f\|_{L^2((k-1) h_N, k h_N; \bm L^2(\Omega))}^2 - (B_M\bm u_k, \bm u_k^*)h_N + (g_k - g_{k-1}, |\nabla \bm u_{k-1}|),
	\end{aligned}
	\end{equation}
	where we have used $\|\bm f_k\|_{\bm L^2(\Omega)} \le (1/\sqrt{h_N}) \|\bm f\|_{L^2((k-1) h_N, k h_N; \bm L^2(\Omega))}$.
	Since $\bm u_k^* \in \partial \varphi_k(\bm u_k)$ implies $\|\bm u_k\|_{\bm H^2(\Omega)} \le C (\|\bm u_k^*\|_{\bm H_\sigma} + \|g_k\|_{H^1(\Omega)})$, cf.\ \eref{K-estimate}, we find from \eref{bm} that
	\begin{align*}
		|(B_M\bm u_k, \bm u_k^*)| &\le C_M \|\bm u_k\|_{\bm H^2(\Omega)}^{1/2} \|\bm u_k^*\|_{\bm H_\sigma} \le C_M (\|\bm u_k^*\|_{\bm H_\sigma}^{3/2} + \|g_k\|_{H^1(\Omega)}^{1/2} \|\bm u_k^*\|_{\bm H_\sigma}) \\
			&\le \frac14 \|\bm u^*_k \|_{\bm H_\sigma}^2 + C_M \Big( \frac1{\sqrt{h_N}} \|g\|_{L^2((k-1)h_N, kh_N; H^1(\Omega))} + 1 \Big),
	\end{align*}
	where we have used $\|g_k\|_{H^1(\Omega)} \le (1/\sqrt{h_N}) \|g\|_{L^2((k-1)h_N, kh_N; H^1(\Omega))}$.
	Finally, by direct computation one has
	\begin{align*}
		(g_k - g_{k-1}, |\nabla \bm u_{k-1}|) &= \frac1{h_N} \int_{(k-1)h_N}^{kh_N} \int_{\min\{0, s-h_N\}}^s (g'(r), |\nabla \bm u_{k-1}|) \, dr \, ds \\
			&\le \|g'\|_{L^1(\min\{(k-2)h_N, 0\}, kh_N; L^2(\Omega))} \|\nabla \bm u_{k-1}\|_{\bm L^2(\Omega)}.
	\end{align*}
	Substituting the two estimates above into \eref{eq2: proof of 2ndest} and absorbing $\|\bm u^*_k \|_{\bm H_\sigma}^2 h_N$ to the left-hand side, we see that
	\begin{align*}
		&\frac\nu2 (\|\nabla \bm u_k\|_{L^2(\Omega)}^2 - \|\nabla \bm u_{k-1}\|_{L^2(\Omega)}^2) + (g_k, |\nabla \bm u_k|) - (g_{k-1}, |\nabla \bm u_{k-1}|) + \frac12 \|\bm u^*_k \|_{\bm H_\sigma}^2 h_N \\
		\le \; &C_M (\|\bm f\|_{L^2((k-1) h_N, k h_N; \bm L^2(\Omega))}^2 + \|g\|_{L^2((k-1)h_N, kh_N; H^1(\Omega))} h_N^{1/2} + h_N ) \\
			&\qquad + \|g'\|_{L^1(\min\{(k-2)h_N, 0\}, kh_N; L^2(\Omega))} \|\nabla \bm u_{k-1}\|_{\bm L^2(\Omega)}.
	\end{align*}
	
	Now, if $m := \operatorname{argmax}_{k=0, \dots, N} \|\nabla \bm u_k\|_{\bm L^2(\Omega)} \ge 1$, then we add the above estimate for $k = 1, \dots, m$ to obtain
	\begin{equation} \label{eq1: proof of 2ndest}
	\begin{aligned}
		&\frac\nu2 (\|\nabla \bm u_m\|_{L^2(\Omega)}^2 - \|\nabla \bm u_0\|_{L^2(\Omega)}^2) + (g_m, |\nabla \bm u_m|) - (g_0, |\nabla \bm u_0|) + \frac12 \sum_{k=1}^m \|\bm u^*_k \|_{\bm H_\sigma}^2 h_N \\
			\le \; &C_M (\|\bm f\|_{L^2(0, T; \bm L^2(\Omega))}^2 + \|g\|_{L^2(0, T; H^1(\Omega))}^2 + T ) + \|g'\|_{L^1(0, T; L^2(\Omega))} \|\nabla \bm u_{m}\|_{\bm L^2(\Omega)} \\
			\le \; &C_M (\|\bm f\|_{L^2(0, T; \bm L^2(\Omega))}^2 + \|g\|_{L^2(0, T; H^1(\Omega))}^2 + \|g'\|_{L^1(0, T; L^2(\Omega))}^2 + T ) + \frac\nu4 \|\nabla \bm u_{m}\|_{\bm L^2(\Omega)}^2,
	\end{aligned}
	\end{equation}
	which gives us a bound of $\max_{k=0, \dots, N} \|\nabla \bm u_k\|_{\bm L^2(\Omega)}$ independent of $N$ (note: $\|g_0\|_{L^2(\Omega)} \le C(\Omega, T) \|g\|_{W^{1,1}(0, T; L^2(\Omega))}$).
	A similar computation to \eref{eq1: proof of 2ndest}, in which $m$ is replaced by an arbitrary integer, proves \eref{eq: 2ndest} for the first 3 terms on the left-hand side.
	The fourth and fifth terms can be addressed by
	\begin{align*}
		\|\bm u_k\|_{\bm H^2(\Omega)}^2 &\le C \Big( \|\bm u_k^*\|_{\bm H_\sigma}^2 + \frac1{h_N} \|g\|_{L^2((k-1)h_N, kh_N; H^1(\Omega))}^2 \Big), \\
		\frac{\bm u_k-\bm u_{k-1}}{h_N} &= P \bm f_k - \bm u_k^* - B_M \bm u_k, \quad \|B_M \bm u_k\|_{\bm H_\sigma}^2 \le C_M (\|\bm u_k^*\|_{\bm H_\sigma} + \|g_k\|_{H^1(\Omega)}),
	\end{align*}
	which completes the proof of the lemma.
\end{proof}

Finally we derive an $H^1(0, T; \bm H^1(\Omega))$-like a priori estimate for $\{\bm u_k\}_{k=1}^N$ to show \tref{thm: evolution}.
\begin{lem} \label{3rdest}
	Under the assumptions of \tref{thm: evolution}, let $M \ge \|\bm u_0\|_{\bm V_\sigma}$.
	Then, for $m = 1, \dots, N$ we have
	\begin{align*}
		\left\| \frac{\bm u_{m}-\bm u_{m-1}}{h_N} \right\|_{\bm H_\sigma}^2
			+ \sum_{k=1}^m \left\|\nabla \frac{\bm u_{k} - \bm u_{k-1}}{h_N} \right\|_{\bm L^2(\Omega)}^2 h_N
		&\le \; C_M (\|\bm u_0\|_{\bm V_\sigma}^2 + \|\bm f\|_{W^{1,1}(0, T; \bm L^2(\Omega))}^2 + \|g\|_{H^1((0, T) \times \Omega)}^2 + \|\bm f^*\|_{\bm H_\sigma}^2),
	\end{align*}
	where the constant $C_M = C(\Omega, \nu, T, M)$ is independent of $m$ and $N$.
\end{lem}
\begin{proof}
	For the functional $j_2$ defined in the Step 2 of the proof of \lref{lem: maximal monotone}, we see that $\partial j_2 = - (\nu/2) P \Delta$ and that $F := (1/4) \partial j_2 + B_M + \alpha_M I$ is monotone. 
	In addition, observe that \eref{vf} is rewritten as
	\begin{equation} \label{eq2: 3rdest}
	\begin{aligned}
		&\Big( \frac{\bm u_k - \bm u_{k-1}}{h_N}, \bm u_k - \bm v \Big) + \frac{7\nu}{8} (\nabla \bm u_k, \nabla (\bm u_k - \bm v)) + (F\bm u_k, \bm u_k - \bm v) + (g_k, |\nabla \bm u_k| - |\nabla \bm v|) \\
			\le \; &(\bm f_k, \bm u_k - \bm v) + \alpha_M (\bm u_k, \bm u_k - \bm v) \qquad \forall \bm v \in \bm V_\sigma, \quad k = 1, \dots, N.
	\end{aligned}
	\end{equation}
	Taking $\bm v = \bm u_{k-1}$ in \eref{eq2: 3rdest}, taking $\bm v = \bm u_k$ in \eref{eq2: 3rdest} with $k$ replaced by $k-1$, adding the two inequalities, and applying the monotonicity of $F$, we obtain for $k=2, \dots, N$
	\begin{align*}
		&\frac{\|\bm u_k - \bm u_{k-1}\|_{\bm H_\sigma}^2 - \|\bm u_{k-1} - \bm u_{k-2}\|_{\bm H_\sigma}^2}{2h_N} + \frac{\|\bm u_k - \bm u_{k-2}\|_{\bm H_\sigma}^2}{2h_N} + \frac{7\nu}{8} \|\nabla (\bm u_k - \bm u_{k-1})\|_{\bm L^2(\Omega)}^2 \\
		\le \; &(\bm f_k - \bm f_{k-1}, \bm u_k - \bm u_{k-1}) + \alpha_M \|\bm u_k - \bm u_{k-1}\|_{\bm H_\sigma}^2 - (g_k - g_{k-1}, |\nabla \bm u_k| - |\nabla \bm u_{k-1}|) \\
		\le \; & \|\bm f_k - \bm f_{k-1}\|_{\bm L^2(\Omega)} \|\bm u_k - \bm u_{k-1}\|_{\bm H_\sigma} + \alpha_M \|\bm u_k - \bm u_{k-1}\|_{\bm H_\sigma}^2 + C \|g_k - g_{k-1}\|_{L^2(\Omega)}^2 + \frac{3\nu}{8} \|\nabla (\bm u_k - \bm u_{k-1})\|_{\bm L^2(\Omega)}^2.
	\end{align*}
	Therefore,
	\begin{equation} \label{eq1: 3rdest}
	\begin{aligned}
		&\Big\| \frac{\bm u_k - \bm u_{k-1}}{h_N} \Big\|_{\bm H_\sigma}^2 - \Big\| \frac{\bm u_{k-1} - \bm u_{k-2}}{h_N} \Big\|_{\bm H_\sigma}^2 + \nu \Big\| \nabla \frac{\bm u_k - \bm u_{k-1}}{h_N} \Big\|_{\bm L^2(\Omega)}^2 h_N \\
			\le \; & 2 \|\bm f'\|_{L^1((k-2)h_N, kh_N; \bm L^2(\Omega))}\Big\| \frac{\bm u_k - \bm u_{k-1}}{h_N} \Big\|_{\bm H_\sigma} + 2\alpha_M \Big\| \frac{\bm u_k - \bm u_{k-1}}{h_N} \Big\|_{\bm H_\sigma}^2 h_N + 2C \|g'\|_{L^2((k-2)h_N, kh_N; L^2(\Omega))}^2,
	\end{aligned}
	\end{equation}
	where $\|\bm f_k - \bm f_{k-1}\|_{\bm L^2(\Omega)} \le \|\bm f'\|_{L^1((k-2)h_N, kh_N; \bm L^2(\Omega))}$ and $\|g_k - g_{k-1}\|_{\bm L^2(\Omega)} \le \sqrt{h_N} \|g'\|_{L^2((k-2)h_N, kh_N; \bm L^2(\Omega))}$ are used.
	Now we add \eref{eq1: 3rdest} for $k = 2, \dots, m$, argue similarly as in \eref{eq1: proof of 2ndest}, and then substitute \eref{eq: 2ndest} to arrive at
	\begin{equation} \label{eq3: 3rdest}
	\begin{aligned}
		&\Big\| \frac{\bm u_m - \bm u_{m-1}}{h_N} \Big\|_{\bm H_\sigma}^2 + \nu \sum_{k=2}^m \Big\| \nabla \frac{\bm u_k - \bm u_{k-1}}{h_N} \Big\|_{\bm L^2(\Omega)}^2 h_N \\
			\le \; &\Big\| \frac{\bm u_1 - \bm u_{0}}{h_N} \Big\|_{\bm H_\sigma}^2 + C_M (\|\bm u_0\|_{\bm V_\sigma}^2 + \|\bm f\|_{W^{1,1}(0, T; \bm L^2(\Omega))}^2 + \|g\|_{H^1((0, T) \times \Omega)}^2) \qquad (m=2, \dots, N).
	\end{aligned}
	\end{equation}
	
	On the other hand, recall that \eref{initialc} holds, in which we have $B = B_M$ since $M \ge \|\bm u_0\|_{\bm V_\sigma}$.
	Then it follows that (recall that $\bm f_0$ means $\bm f(0)$)
	\begin{equation*}
		\frac{7\nu}{8} (\nabla \bm u_0, \nabla (\bm u_0 - \bm v)) + (F\bm u_0, \bm u_0 - \bm v) + (g_0, |\nabla \bm u_0| - |\nabla \bm v|) \le (\bm f_0 + \bm f^* + \alpha_M \bm u_0, \bm u_0 - \bm v) \quad \forall \bm v \in \bm V_\sigma.
	\end{equation*}
	Taking $\bm v = \bm u_1$ above, taking $\bm v = \bm u_0$ in \eref{eq2: 3rdest} for $k = 1$, adding the two inequalities, and dividing the both sides by $h_N$, we derive (recall $N$ is chosen so large that $h_N \alpha_M \le 1/4$)
	\begin{equation} \label{eq4: 3rdest}
	\begin{aligned}
		&\Big\|\frac{\bm u_1 - \bm u_0}{h_N} \Big\|_{\bm H_\sigma}^2 + \frac{7\nu}{8} \Big\| \nabla \frac{\bm u_1 - \bm u_0}{h_N} \Big\|_{\bm L^2(\Omega)}^2 h_N \\
			\le \;& \Big( \bm f_1 - \bm f_0 - \bm f^*, \frac{\bm u_1 - \bm u_0}{h_N} \Big) + h_N \alpha_M \Big\|\frac{\bm u_1 - \bm u_0}{h_N} \Big\|_{\bm H_\sigma}^2 - \frac1{h_N} (g_1 - g_0, |\nabla \bm u_1| - |\nabla \bm u_0|) \\
			\le \;& C ( \|\bm f'\|_{L^1(0, h_N; \bm L^2(\Omega))}^2 + \|\bm f^*\|_{\bm H_\sigma}^2) + \frac12 \Big\|\frac{\bm u_1 - \bm u_0}{h_N} \Big\|_{\bm H_\sigma}^2
				+ C \|g'\|_{L^2(0, h_N; \bm L^2(\Omega))}^2 + \frac{3\nu}{8} \Big\| \nabla \frac{\bm u_1 - \bm u_0}{h_N} \Big\|_{\bm L^2(\Omega)}^2 h_N.
	\end{aligned}
	\end{equation}
	Combining \eref{eq3: 3rdest} and \eref{eq4: 3rdest} concludes the desired estimate.
\end{proof}

\subsection{Limiting procedure as $N \to \infty$} \label{subsec: limit in N}
We continue to deal with $\{\bm u_k\}_{k=0}^N$ constructed in the previous subsection.
Let us introduce a piecewise-linear interpolation $\hat{\bm u}_N:[0,T] \to \bm H_\sigma$ and piecewise-constant interpolations $\bar{\bm u}_N:(0,T] \to \bm H_\sigma$ and $\bar{\bm f}_N:(0,T] \to \bm H_\sigma$ as follows: for $k=1, \dots, N$,
\begin{alignat*}{2}
	\hat{\bm u}_N(t)
	& = \frac{kh_N - t}{h_N} \bm u_{k-1} + \frac{t - (k-1)h_N}{h_N}\bm u_{k}
	\qquad && {\rm for~} t \in [(k-1)h_N, kh_N ], \\
	\bar{\bm u}_N(t) 
	& = \bm u_{k}, 
	\quad 
	\bar{\bm f}_N(t) =\bm f_k \qquad && {\rm for~} t \in ( (k-1) h_N, kh_N ].
\end{alignat*}
Define $\Phi_N, \Phi:L^2(0,T; \bm V_\sigma) \to \mathbb{R}$ by  
\begin{equation*}
	\Phi_N(\bm \upsilon ) = \sum_{k=1}^N \int_{(k-1)h_N}^{k h_N} \varphi_k ( \bm \upsilon (t) ) \, dt, \qquad
	\Phi(\bm \upsilon ) = \int_0^T \varphi^t ( \bm \upsilon (t) ) \, dt,
\end{equation*}
where $\varphi^t := j_{g(t)}$, that is, $\varphi^t(\bm v) = (\nu/2)\|\nabla \bm v\|_{\bm L^2(\Omega)}^2 + (g(t), |\nabla \bm v|)$ for each $t \in [0, T]$.
By direct computation, we see that $\Phi_N(\bm \upsilon ) \to \Phi(\bm \upsilon )$ as $N \to \infty$ for $\bm \upsilon  \in L^2(0,T;\bm V_\sigma)$
and that $|\Phi_N(\bm \upsilon_1) - \Phi_N(\bm \upsilon_2)| \le \|g\|_{L^2(0, T; L^2(\Omega))} \|\bm\upsilon_1 - \bm\upsilon_2\|_{L^2(0, T; \bm V_\sigma)}$.

Now we observe that \eref{vf} can be reformulated as the following space-time variational inequality:
\begin{equation} \label{0Tvf}
\begin{aligned}
	\int_0^T ( 
	\hat{\bm u}_N'(t), \bm \upsilon(t)-
	\bar{\bm u}_N(t) 
	) dt 
	+ \int_0^T 
	( B_M \bar{\bm u}_N(t), \bm \upsilon(t)-
	\bar{\bm u}_N(t) 
	) dt
	+ \Phi_N (\bm \upsilon)
	- \Phi_N ( \bar{\bm v}_N )
	\ge 
	\int_0^T ( 
	\bar{\bm f}_{N}(t), \bm \upsilon(t)-
	\bar{\bm u}_N(t) 
	) dt
\end{aligned}
\end{equation}
for all $\bm \upsilon \in L^2(0,T;\bm V_\sigma)$.
From this we construct a solution of \eref{wf} in which $B$ is replaced with its truncation $B_M$.
The following result is an extension of \cite[Proposition 5.15]{Bar10} to a time-dependent nonlinear multivalued operator $\partial \varphi^t$.
\begin{prop} \label{prop: sub}
	Under the assumptions in \tref{thm: evolution}, there exists at least one $\bm u \in W^{1,\infty}(0,T; \bm H_\sigma)\cap H^1(0,T; \bm V_\sigma) \cap L^\infty (0,T; \bm W)$ such that $\bm u(0)=\bm u_0$ and 
	\begin{equation} \label{vfa}
		 (\bm u'(t), \bm v -\bm u(t)) - \nu (\Delta \bm u(t), \bm v -\bm u(t)) + ( B_M \bm u(t), \bm v - \bm u(t)) + (g(t), |\nabla \bm v| - |\nabla \bm u(t)| ) \ge ( \bm f(t), \bm v - \bm u(t) )
	\end{equation}
	for all $\bm v \in \bm V_\sigma$ and a.a.\ $t \in (0, T)$.
\end{prop}

\begin{proof}
Collecting the estimates in Lemmas \ref{1stest}, \ref{2ndest}, and \ref{3rdest}, we deduce that, for sufficiently large $N$,
\begin{align}
	\|\hat{\bm u}_N' \|_{L^\infty(0,T;\bm H_\sigma)} + \|\hat{\bm u}_N' \|_{L^2(0,T;\bm V_\sigma)}
		+ \|\hat{\bm u}_N \|_{L^\infty (0,T;\bm V_\sigma)} + \| \hat{\bm u}_N \|_{L^2(0,T;\bm H^2(\Omega))} &\le C_M, \\
	\|\bar{\bm u}_N \|_{L^\infty (0,T;\bm V_\sigma)} + \|\bar{\bm u}_N \|_{L^2(0,T;\bm H^2(\Omega))} &\le C_M, \\
	\| \hat{\bm u}_N - \bar{\bm u}_N \|_{L^2(0,T;\bm V_\sigma)}^2 = \frac{h_N^{2}}{3} \sum _{k=1}^N \Big\| \frac{\bm u_k - \bm u_{k-1}}{h_N} \Big\|_{\bm V_\sigma}^2 h_N \le \frac{h_N^{2}}{3} C_M \label{1/3} \to 0 \quad \text{as $N \to \infty$},
\end{align}
where $C_M = C(\Omega, \nu, T, \bm f, g, \bm u_0, \bm f^*, M)$ is independent of $N$.
As a result of the weak compactness and of the Aubin--Lions compactness theorem \cite[Corollary 4]{Sim87}, we can find some $\bm u \in W^{1,\infty}(0,T;\bm H_\sigma) \cap H^1 (0,T;\bm V_\sigma) \cap L^2(0,T;\bm W)$, $\bar{\bm u} \in L^\infty (0,T;\bm V_\sigma) \cap L^2(0,T;\bm W)$, and subsequence $\{N_n \}_{n \in \mathbb{N}}$ such that
\begin{align}
	\hat{\bm u}_{N_n}' \to \bm u'
	& \quad {\rm weakly~star~in~} L^\infty(0,T;\bm H_\sigma) \cap L^2 (0,T;\bm V_\sigma ), \notag \\
	\hat{\bm u}_{N_n} \to \bm u
	& \quad {\rm weakly~star~in~} L^\infty(0,T;\bm V_\sigma) \cap L^2 (0,T;\bm W ), \notag \\
	\hat{\bm u}_{N_n} \to \bm u 
	& \quad {\rm strongly~in~} C( [0,T];\bm H_\sigma ) \cap L^2(0,T;\bm V_\sigma), \label{eq: strong conv hat} \\
	\bar{\bm u}_{N_n} \to \bar{\bm u}
	& \quad {\rm weakly~star~in~} L^\infty(0,T;\bm V_\sigma)\cap L^2 (0,T;\bm W ), \notag
\end{align}
as $n \to \infty$.
By virtue of \eref{1/3} and \eref{eq: strong conv hat}, one has
\begin{equation*}
	\| \bar {\bm u}_{N_n}-\bm u \|_{L^2(0,T;\bm V_\sigma)}
	\le \| \bar {\bm u}_{N_n}-\hat{\bm u}_{N_n} \|_{L^2(0,T;\bm V_\sigma)}
	+\| \hat{\bm u}_{N_n}-\bm u \|_{L^2(0,T;\bm V_\sigma)}
	\to 0 \quad \text{as $n \to \infty$},
\end{equation*}
so that $\bar{\bm u}_{N_n} \to \bm u$ strongly in $L^2(0, T; \bm V_\sigma)$ and $\bm u = \bar{\bm u}$.
This implies that $B_M \bar{\bm u}_{N_n} \to B_M \bm u$ strongly in $L^2(0, T; \bm V_\sigma')$ in view of \eref{contibm} and that
\begin{equation*}
	| \Phi_{N_n} (\bar{\bm u}_{N_n}) - \Phi (\bm u) | 
	\le | \Phi_{N_n} (\bar{\bm u}_{N_n}) - \Phi _{N_n}(\bm u) |+
	| \Phi_{N_n} (\bm u) - \Phi (\bm u) | \\
	\to 0 \quad \text{as $n \to \infty$}.
\end{equation*}   

Therefore, setting $N = N_n$ and letting $n \to +\infty$ in \eqref{0Tvf}, we obtain, for all $\bm \upsilon \in L^2(0,T;\bm V_\sigma)$,
\begin{gather*}
	\int_0^T ( 
	\bm u'(t), \bm \upsilon(t)-
	\bm u(t) 
	) dt 
	+ \int_0^T 
	( B_M \bm u(t), \bm \upsilon(t)-
	\bm u(t) 
	) dt
	+ \Phi(\bm \upsilon )
	- \Phi(\bm u)
	\ge 
	\int_0^T ( 
	\bm f(t), \bm \upsilon(t)-
	\bm u(t)
	) dt.
\end{gather*}
In particular, $\bm u(t) \in D(\partial \varphi^t) \subset \bm W$ for a.a.\ $t \in (0,T)$ and the following evolution equation holds:
\begin{gather}
	\begin{cases}
	\bm u'(t)
	+ \partial \varphi^t ( \bm u(t))
	+ B_M \bm u(t)
	\ni P \bm f(t)
	\quad \mbox{in~}\bm H_\sigma, \quad \mbox{for~a.a.~} t \in (0,T),\\
	\bm u(0) = \bm u_0
	\quad \mbox{in~} \bm H_\sigma.
	\end{cases}
	\label{ee}
\end{gather}
This $\bm u$ satisfies the equivalent variational formulation \eqref{vfa}.  
Finally, using \eref{bm}, \eref{ee}, and \eref{K-estimate}, we have
\begin{align*}
	\|\bm u(t)\|_{\bm H^2(\Omega)} &\le C(\|P\bm f(t) - \bm u'(t) - B_M \bm u(t)\|_{\bm H_\sigma} + \|g(t)\|_{H^1(\Omega)}) \\
		&\le C_M (\|\bm f(t)\|_{\bm L^2(\Omega)} + \|\bm u'(t)\|_{\bm H_\sigma} + \|g(t)\|_{H^1(\Omega)} + 1) + \frac12 \|\bm u(t)\|_{H^2(\Omega)} \quad \text{for a.a.\ $t \in (0, T)$},
\end{align*}
which yields $\bm u \in L^\infty(0,T;\bm W)$.
This completes the proof of the proposition.
\end{proof}

\subsection{Proof of \tref{thm: evolution}}
Let $\bm u$ be a solution obtained by Proposition \ref{prop: sub}, which means the existence of some $\bm u^* \in L^2(0,T;\bm H_\sigma)$ such that
\begin{equation*}
	\bm u'(t) + \bm u^*(t) + B_M \bm u(t) = P\bm f(t) \text{ in } \bm H_\sigma \quad\text{and}\quad \bm u^*(t) \in \partial \varphi^t(\bm u(t)) \text{ in } \bm H_\sigma \quad \text{for a.a.\ } t \in (0,T).
\end{equation*}
Let us recall a chain rule for a time-dependent convex functional, which is essentially due to \cite[Lemma 2.10]{AO04}, \cite[Theorem 2.1.1]{Ken81}, \cite[Proposition 3.4]{Ota82}, \cite[Lemma 2.4]{Ota93}, and \cite[Lemma 3.6]{OY78}.

\begin{lem}\label{lem: chain}
The mapping $t \mapsto \varphi^t(\bm u(t))$ is differentiable a.e.\ on $(0,T)$, and it satisfies
\begin{equation}
	\frac{d}{dt} \varphi^t ( \bm u(t) )
	- ( \bm u^*(t), \bm u'(t) )
	\le \| g'(t) \|_{L^2(\Omega)} \| \nabla \bm u (t) \|_{\bm L^2(\Omega)}
	\quad \mbox{for a.a.\ } t \in (0,T). 
	\label{chain}
\end{equation}
\end{lem}

\begin{proof}
An argument based on the Moreau--Yosida regularization (see e.g.\ \cite[p.\ 29]{Ken81}) shows that $t \mapsto \varphi^t(\bm u(t))$ is absolutely continuous on $[0,T]$, which implies its differentiability.
Since $\bm u' \in L^2(0,T;\bm H_\sigma)$, we have
\begin{equation*}
	\frac{\bm u(t+h)-\bm u(t)}{h} = \frac{1}{h} \int_{t}^{t+h} \bm u'(\tau) d\tau 
	\to \bm u'(t) \quad {\rm in~} \bm H_\sigma \text{ as $h \to 0$, for a.a.\ $t \in (0, T)$}.
\end{equation*}
Next, since $\bm u^* \in L^2(0,T; \bm H_\sigma)$, for a.a.\ $t \in (0,T)$ there exists $\{h_n\}$, $h_n \searrow 0$ such that
\begin{equation*}
	\bm u^*(t+h_n) \in \partial \varphi^{t+h_n} ( \bm u(t+h_n)), \quad 
	\bm u^*(t+h_n) \to \bm u^*(t) \quad {\rm in~} \bm H_\sigma \quad \text{as $n \to \infty$}.
\end{equation*}
By the definition of the subdifferential of $\varphi^{t+h_n}$,
\begin{align*}
	\varphi^{t+h_n} ( \bm u(t+h_n) ) - \varphi^{t} ( \bm u(t) ) - ( \bm u(t+h_n)-\bm u(t), \bm u^*(t+h_n) )
	& \le \varphi^{t+h_n} ( \bm u(t) ) - \varphi^ t ( \bm u(t) ) \\
	& \le \int_t^{t+h_n} \| g'(\tau) \|_{L^2(\Omega)} d\tau \| \nabla \bm u(t) \|_{\bm L^2(\Omega)},
\end{align*}
which proves \eref{chain} after dividing by $h_n$ and taking $n \to \infty$.
\end{proof}

Let us start the proof of the existence part in \tref{thm: evolution}.
For this purpose, it suffices to establish an upper bound of $\|\bm u(t)\|_{\bm V_\sigma}$ that is uniform in $t$ and independent of $M$.
By \lref{1stest}, we have $\|\bm u \|_{L^\infty(0, T; \bm H_\sigma) \cap L^2(0, T; \bm V_\sigma)} \le C(\Omega, \nu, \bm u_0, \bm f)$.
Next we apply \lref{lem: chain} to obtain
\begin{equation*}
	\frac{d}{dt} \varphi^t ( \bm u(t) ) + \|\bm u^*(t) \|_{\bm H_\sigma}^2
		\le ( \bm u^*(t), \bm f(t) - B_M \bm u(t) ) + \| g'(t) \|_{L^2(\Omega)} \| \nabla \bm u (t) \|_{\bm L^2(\Omega)} \quad \text{for a.a.\ $t \in (0,T)$,}
\end{equation*}
where we observe from \eref{bm3}, \eref{eq: coercivity}, \eref{K-estimate}, and Young's inequality that
\begin{align*}
	(\bm u^*(t), \bm f(t) - B_M \bm u(t)) &\le \|\bm u^*(t)\|_{\bm H_\sigma} \big[ \|\bm f(t)\|_{\bm L^2(\Omega)}
		+ C\|\nabla \bm u(t)\|_{\bm L^2(\Omega)}^{3/2} (\|\bm u^*(t)\|_{\bm L^2(\Omega)}^{1/2} + \|g(t)\|_{H^1(\Omega)}^{1/2}) \big] \\
	&\le \frac12 \|\bm u^*(t)\|_{\bm H_\sigma}^2 + C (\|\bm f(t)\|_{\bm L^2(\Omega)}^2 + \|\nabla \bm u(t)\|_{\bm L^2(\Omega)}^6 + \|g(t)\|_{H^1(\Omega)}^2).
\end{align*}
Consequently, for a.a.\ $t \in (0,T)$,
\begin{equation*}
	\frac{d}{ds} \varphi^s ( \bm u(t) )
		\le C (\|\bm f(t)\|_{\bm L^2(\Omega)}^2 + \|g(t)\|_{H^1(\Omega)}^2) + \| g'(t) \|_{L^2(\Omega)} \|\nabla \bm u (t) \|_{\bm L^2(\Omega)} + C\|\nabla \bm u(t)\|_{\bm L^2(\Omega)}^6.
\end{equation*}
If we set $y(t) := \varphi^t ( \bm u(t) ) + 1 \ge 1$, this yields
\begin{equation} \label{eq1: prf of Thm4.1}
	y'(t) \le \underbrace{ C(\|\bm f(t)\|_{\bm L^2(\Omega)}^2 + \|g(t)\|_{H^1(\Omega)}^2 + \|g'(t)\|_{L^2(\Omega)} + \|\nabla \bm u (t) \|_{\bm L^2(\Omega)}^2) }_{ =: a(t) } y(t)^2, \quad C = C(\Omega, \nu),
\end{equation}
where $\int_0^T a(s) \, ds \le C(\Omega, \nu, \bm f, g, \bm u_0)$ independently of $M$ by \lref{1stest}. Then, for suitably small $t > 0$,
\begin{equation*}
	y(t) \le \frac{y(0)}{1 - y(0) \int_0^t a(s) \, ds}.
\end{equation*}

It is possible to choose $T^* = C(\Omega, \nu, \bm f, g, \bm u_0)$ in such a way that $y(0) \int_0^{T^*} a(t) \, dt \le 1/2$, which gives us a bound $\|\bm u\|_{L^\infty(0, T^*; \bm V_\sigma)} \le C(\Omega, \nu, \bm f, g, \bm u_0)$ independently of $M$.
We remark that in case $d = 2$ we need not restrict $T$ to $T^*$.
In fact, use \eref{bm2} in place of \eref{bm3} to arrive at a similar estimate to \eref{eq1: prf of Thm4.1}, where the power of $y(t)$ in the right-hand side is replaced by 1.
Then Gronwall's inequality enables us to bound $\bm u \in L^\infty(0, T; \bm V_\sigma)$ globally in $(0, T)$.

Now fix $M$ to be greater than the $L^\infty(0, T^*; \bm V_\sigma)$-bound above, and consider $\bm u$ constructed in \pref{prop: sub}.
Since $B_M\bm u(t) = B\bm u(t)$ for a.a.\ $t \in (0, T^*)$, this $\bm u$ indeed solves \eref{wf} on $(0, T^*)$, which proves the existence part.

The uniqueness part of \tref{thm: evolution} is easy.
In fact, let $\bm v$ and $\bm u$ be two solutions of \eref{wf}.
Then, we take $\bm v(t)$ (resp.\ $\bm u(t)$) in the variational inequality that $\bm u(t)$ (resp.\ $\bm v(t)$) satisfies, and add the resulting inequalities to get
\begin{align*}
	\frac{1}{2} \frac{d}{dt} \| \bm v(t)-\bm u(t) \|_{\bm H_\sigma}^2
	+ \nu \| \nabla ( \bm v(t)-\bm u(t) ) \|^2_{\bm L^2(\Omega)}
	& \le | b(\bm v(t),\bm v(t),\bm v(t)-\bm u(t) ) - 
	b(\bm u(t),\bm u(t),\bm v(t)-\bm u(t) ) | \\
	& =  | b(\bm v(t)-\bm u(t),\bm v(t)-\bm u(t), \bm u(t) ) |.
\end{align*}
Then a standard argument, e.g.\ \cite[Chapter 3]{Tem01}, leads to $\bm v = \bm u$.
This completes the proof of \tref{thm: evolution}. 

\subsection{Proof of \tref{thm: evolution weak}}
We exploit the same strategy as in \tref{thm: evolution}, however without Lemma \ref{3rdest} this time (Lemmas \ref{1stest} and \ref{2ndest} remain to hold in the setting of \tref{thm: evolution weak}).
The difficulty is two folded: 1) the strong convergence of $\hat{\bm u}_{N}$ in $L^2(0, T; \bm V_\sigma)$ is not trivial if we do not know the boundedness of $\hat{\bm u}_N' \in L^2(0, T; \bm V_\sigma)$; 2) $\hat{\bm u}_{N}$ may not belong to $L^2(0, h_N; \bm H^2(\Omega))$ because we only assume $\bm u_0 \in \bm V_\sigma$.

To address issue 1), we devise an auxiliary estimate for \lref{2ndest}.
\begin{lem} \label{4thest}
	We have
	\begin{equation} \label{eq: 4thest}
	\begin{aligned}
		\sum_{k=1}^N \|\nabla(\bm u_k - \bm u_{k-1})\|_{\bm L^2(\Omega)}^2 \le C_M (\|\nabla\bm u_0\|_{\bm L^2(\Omega)}^2 + \|\bm f\|_{L^2(0,T;\bm L^2(\Omega))}^2 + \|g\|_{L^2(0,T; H^1(\Omega))}^2 + \|g\|_{W^{1,1}(0,T;L^2(\Omega))}^2),
	\end{aligned}
	\end{equation}
	where the constant $C_M = C(\Omega, \nu, T, M)$ is independent of $N$.
\end{lem}

\begin{proof}
For $k=1, \dots, N$, we take $\bm v = \bm u_{k-1}$ in \eqref{vf} to get
\begin{align*}
& \left\| \frac{\bm u_k-\bm u_{k-1}}{h_N} \right\|_{\bm H_\sigma}^2 h_N
+ \frac{\nu}{2} (\| \nabla \bm u_k \|_{\bm L^2(\Omega)}^{2} - \| \nabla \bm u_{k-1}\|_{\bm L^2(\Omega)}^2
+ \| \nabla( \bm u_k-\bm u_{k-1}) \|_{\bm L^2(\Omega)}^2 )
+ (g_k, |\nabla \bm u_{k}|)-(g_{k-1}, |\nabla \bm u_{k-1}|) \\
\le \; &-(B_M \bm u_k - \bm f_{k}, \bm u_{k}- \bm u_{k-1} ) + ( g_k - g_{k-1}, |\nabla \bm u_{k-1}|).
\end{align*}
Adding this for $k=1, \dots, N$ and applying the estimates known from \lref{2ndest}, we obtain \eref{eq: 4thest}.
\end{proof}

Next, to resolve issue 2), we introduce another piecewise linear interpolation of $\{\bm u _k\}$ by
\begin{equation*}
	\bm w_N(t) = 
	\begin{cases}
	\bm u_1 & \quad {\rm if~} t \in [0, h_N], \\
	\hat {\bm u}_N (t) & \quad {\rm if~} t \in [h_N, T].
	\end{cases}
\end{equation*}
Then we have $\bm w_N (t) \in \bm H^2(\Omega)$ for all $t \in [0,T]$ and
\begin{align*}
	\| \bm w _N' \|_{L^2(0,T;\bm L^2(\Omega))}^2 &\le \| \hat {\bm u}_N' \|_{L^2(0,T;\bm L^2(\Omega))}^2, \\
	\| \bm w _N \|_{L^2(0,T;\bm H^2(\Omega))}^2 &\le \frac{h_N}{2} \| \bm u_1 \|_{\bm H^2(\Omega)} ^2 + \| \hat{\bm u}_N \|_{L^2(0,T;\bm H^2(\Omega))}^2,
\end{align*}
which implies $\|\bm w_N\|_{H^1(0,T; \bm H_\sigma) \cap L^2(0, T; \bm W)} \le C(\Omega, \nu, T, \bm f, g, \bm u_0, M)$ by \lref{2ndest}.
Combining this with the boundedness of $\{\hat {\bm u}_N\}$ and $\{\bar{\bm u}_N\}$ already known, we can find some $\bm u \in H^1(0,T;\bm H_\sigma) \cap L^\infty (0,T;\bm V_\sigma)$, $\bar{\bm u} \in L^\infty(0,T; \bm V_\sigma) \cap L^2(0,T; \bm W)$, $\bm w \in H^1(0,T; \bm H_\sigma) \cap L^2(0,T;\bm W)$ and subsequence $\{N_n \}_{n \in \mathbb{N}}$ such that
\begin{align*}
	\hat{\bm u}_{N_n} \to \bm u
	& \quad {\rm weakly~star~in~} H^1(0,T;\bm H_\sigma) \cap L^\infty(0,T;\bm V_\sigma), \\
	\hat{\bm u}_{N_n} \to \bm u
	& \quad {\rm strongly~in~} L^2(0,T; \bm H_\sigma), \\
	\bar{\bm u}_{N_n} \to \bar{\bm u}
	& \quad {\rm weakly~star~in~} L^\infty(0,T;\bm V_\sigma)\cap L^2 (0,T;\bm W ), \\
	\bm w_{N_n} \to \bm w
	& \quad \text{weakly in } H^1(0,T; \bm H_\sigma) \cap L^2(0,T;\bm W), \\
	\bm w_{N_n} \to \bm w
	& \quad \text{strongly in } C([0,T]; \bm H_\sigma) \cap L^2(0,T; \bm V_\sigma),
\end{align*}
when $n \to \infty$, as a result of weak compactness and the Aubin--Lions theorem.

Now we claim $\bm u = \bar{\bm u} = \bm w$. In fact, since
\begin{equation*}
	\|\bar{\bm u}_N - \hat{\bm u}_N\|_{L^2(0, T; \bm H_\sigma)}^2 = \frac{h_N}3 \sum_{k=1}^N \|\bm u_k - \bm u_{k-1}\|_{\bm H_\sigma}^2 \to 0 \quad (N \to \infty)
\end{equation*}
by direct computation and \lref{1stest}, we have
\begin{equation*}
	\|\bar{\bm u}_{N_n} - {\bm u}\|_{L^2(0, T; \bm H_\sigma)} \le \|\bar{\bm u}_{N_n} - \hat{\bm u}_{N_n}\|_{L^2(0, T; \bm H_\sigma)} + \|\hat{\bm u}_{N_n} - \bm u\|_{L^2(0, T; \bm H_\sigma)} \to 0 \quad (n \to \infty),
\end{equation*}
hence $\bm u = \bar{\bm u}$.
A similar computation based on $\|{\bm w}_{N} - \hat{\bm u}_{N}\|_{L^2(0, T; \bm H_\sigma)}^2 = (h_N/3) \|\bm u_1 - \bm u_{0}\|_{\bm H_\sigma}^2$ shows $\bm u = \bm w$.

We note that these counterparts for the $\bm H^1(\Omega)$-norm also hold, that is,
\begin{equation*}
	\|\nabla(\bar{\bm u}_N - \hat{\bm u}_N)\|_{L^2(0, T; \bm L^2(\Omega))}^2 = \frac{h_N}3 \sum_{k=1}^N \|\nabla(\bm u_k - \bm u_{k-1})\|_{\bm L^2(\Omega)}^2 \to 0, \quad
	\|\nabla({\bm w}_N - \hat{\bm u}_N)\|_{L^2(0, T; \bm L^2(\Omega))}^2 \to 0 \quad (N \to \infty)
\end{equation*}
by \lref{2ndest}. Therefore, we conclude
\begin{equation*}
	\|\bar{\bm u}_{N_n} - \bm u\|_{L^2(0, T; \bm V_\sigma)} \le \|\bar{\bm u}_{N_n} - \hat{\bm u}_{N_n}\|_{L^2(0, T; \bm V_\sigma)}
		+ \|\hat{\bm u}_{N_n} - {\bm w}_{N_n}\|_{L^2(0, T; \bm V_\sigma)}
		+ \|{\bm w}_{N_n} - \bm u\|_{L^2(0, T; \bm V_\sigma)} \to 0 \quad (n \to \infty),
\end{equation*}
which establishes the strong convergence of $\{\bar{\bm u}_{N_n}\}$ in $L^2(0, T; \bm V_\sigma)$.

Now we can proceed completely in the same way as in \pref{prop: sub} and \tref{thm: evolution}, finishing the proof of \tref{thm: evolution weak}.

\section*{Acknowledgments}
The authors are grateful to the anonymous referee for the careful reading of the manuscript and valuable comments that have led to a significant improvement of its quality.
The first author acknowledges the support from the JSPS KAKENHI Grant-in-Aid for Scientific Research(C), Japan, Grant Number 21K03309.
The second author acknowledges the support from the JSPS KAKENHI Grant-in-Aid for Early-Career Scientists and Grant-in-Aid for Scientific Research(C), Japan, Grant Numbers 20K14357 and 24K06860.




\begin{thebibliography}{10}
\expandafter\ifx\csname url\endcsname\relax
  \def\url#1{\texttt{#1}}\fi
\expandafter\ifx\csname urlprefix\endcsname\relax\def\urlprefix{URL }\fi
\expandafter\ifx\csname href\endcsname\relax
  \def\href#1#2{#2} \def\path#1{#1}\fi

\bibitem{AMS2023}
T.~Aiki, D.~Mizuno, K.~Shirakawa, A class of initial-boundary value problems
  governed by pseudo-parabolic weighted total variation flows, Adv.\ Math.\
  Sci.\ Appl. 32 (2023) 311--341.

\bibitem{AO04}
G.~Akagi, M.~\^Otani, Time-dependent constraint problems arising from
  macroscopic critical-state models for type-{II} superconductivity and their
  approximations, Adv.\ Math.\ Sci.\ Appl. 14 (2004) 683--712.

\bibitem{AMS2024}
H.~Antil, D.~Mizuno, K.~Shirakawa, Well-posedness of a pseudo-parabolic {KWC}
  system in materials science, SIAM J. Math.\ Anal. 56 (2024) 6422--6445.

\bibitem{Bar97}
V.~Barbu, Time optimal control of {N}avier--{S}tokes equations, Systems Control
  Lett. 30 (1997) 93--100.

\bibitem{Bar10}
V.~Barbu, Nonlinear Differential Equations of Monotone Types in Banach Spaces,
  Springer, 2010.

\bibitem{BS01}
V.~Barbu, S.~S. Sritharan, Flow invariance preserving feedback controllers for
  the {N}avier--{S}tokes equation, J.\ Math.\ Anal.\ Appl. 255 (2001) 281--307.

\bibitem{BdV04}
H.~{B}eir\~ao~da {V}eiga, Regularity for {S}tokes and generalized {S}tokes
  systems under nonhomogeneous slip-type boundary conditions, Adv. Differential
  Equations 9 (2004) 1079--1114.

\bibitem{BL14}
M.~Boukrouche, G.~{\L}ukaszewicz, On global in time dynamics of a planar
  {B}ingham flow subject to a subdifferential boundary condition, Discrete
  Contin.\ Dyn.\ Syst.\ Ser.\ S 34 (2014) 3969--3983.

\bibitem{Bre1971}
H.~Brezis, Monotonicity methods in {H}ilbert spaces and some applications to
  nonlinear partial differential equations, in: Contributions to Nonlinear
  Functional Analysis, Academic Press, 1971, pp. 101--156.

\bibitem{Bre73}
H.~Brezis, Op\'erateurs Maximaux Monotones et Semi-groupes de Contractions dans
  les Espaces de Hilbert, North-Holland, 1973.

\bibitem{Bre2010}
H.~Brezis, Functional Analysis, Sobolev Spaces and Partial Differential
  Equations, Springer, 2010.

\bibitem{CRK06}
T.~Caraballo, J.~Real, P.~E. Kloeden, Unique strong solutions and
  ${V}$-attractors of a three dimensional system of globally modified
  {N}avier--{S}tokes equations, J.\ Math.\ Anal.\ Appl. 6 (2006) 411--436.

\bibitem{CoFo1989}
P.~Constantin, C.~Foias, Navier--{S}tokes equations, University of Chicago
  Press, 1989.

\bibitem{DD11}
G.~Deugoue, J.~K. Djoko, On the time discretization for the globally modified
  three dimensional {N}avier--{S}tokes equations, J.\ Comput.\ Appl.\ Math. 235
  (2011) 2015--2029.

\bibitem{DuLi1976}
G.~Duvaut, J.~L. Lions, Inequalities in {M}echanics and {P}hysics, Springer,
  1976.

\bibitem{FuSe1998}
M.~Fuchs, G.~Seregin, Regularity results for the quasi-static {B}ingham
  variational inequality in dimensions two and three, Math. Z. 227 (1998)
  525--541.

\bibitem{FuSe2000}
M.~Fuchs, G.~Seregin, Variational {M}ethods for {P}roblems from {P}lasticity
  {T}heory and for {G}eneralized {N}ewtonian {F}luids, Springer, 2000.

\bibitem{Fukao2013}
T.~Fukao, On a variational inequality of {B}ingham and {N}avier--{S}tokes type
  in three dimensional space, in: Nonlinear Analysis in Iinterdisciplinary
  Sciences, Vol.~36 of GAKUTO Internat.\ Ser.\ Math.\ Sci.\ Appl., 2013, pp.
  57--71.

\bibitem{GiRa1986}
V.~Girault, P.-A. Raviart, Finite {E}lement {M}ethods for {N}avier--{S}tokes
  {E}quations, Springer, 1986.

\bibitem{Glo2008}
R.~Glowinski, Numerical {M}ethods for {N}onlinear {V}ariational {P}roblems,
  Springer, 1984.

\bibitem{GLT1981}
R.~Glowinski, J.~L. Lions, R.~Tr\'emoli\`eres, Numerical {A}nalysis of
  {V}ariational {I}nequalities, North-Holland, 1981.

\bibitem{Gonz2020}
S.~Gonz\'{a}lez-Andrade, A {BDF}2-semismooth {N}ewton algorithm for the
  numerical solution of the {B}ingham flow with temperature dependent
  paramerters, J. Non-Newtonian Fluid Mech. 284 (2020) 104380.

\bibitem{Kato1993}
Y.~Kato, Variational inequalities of {B}ingham type in three dimensions, Nagoya
  Math.\ J. 129 (1993) 53--95.

\bibitem{Ken81}
N.~Kenmochi, Solvability of nonlinear parabolic equations with time-dependent
  constraints and applications, Bull.\ Fac.\ Education, Chiba Univ. 30 (1981)
  1--87.

\bibitem{Kim1986}
J.~U. Kim, On the {C}auchy problem associated with the motion of a {B}ingham
  fluid in the plane, Trans.\ Amer.\ Math.\ Soc. 298 (1986) 371--400.

\bibitem{Kim1987}
J.~U. Kim, On the initial-boundary value problem for a {B}ingham fluid in a
  three-dimensional domain, Trans.\ Amer.\ Math.\ Soc. 304 (1987) 751--770.

\bibitem{Kim1989}
J.~U. Kim, A finite element approximation of three dimensional motion of a
  {B}ingham fluid, Math.\ Mod.\ Numer.\ Anal. 23 (1989) 293--333.

\bibitem{MMD09}
B.~Merouani, F.~Messelmi, S.~Drabla, Dynamical flow of a {B}ingham fluid with
  subdifferential boundary condition, Anal.\ Univ.\ Oradea, fasc.\ Math. 16
  (2009) 5--30.

\bibitem{Mes14}
F.~Messelmi, Diffusion-reaction problem for the {B}ingham fluid with
  {L}ipschitz source, Int.\ J.\ Adv.\ Appl.\ Math.\ and Mech. 3 (2014) 37--55.

\bibitem{NW79}
J.~Naumann, M.~Wulst, On evolution inequalities of {B}ingham type in three
  dimensions {I}, J.\ Math.\ Anal.\ Appl. 68 (1979) 211--227.

\bibitem{Ota82}
M.~\^Otani, Nonmonotone perturbations for nonlinear parabolic equations
  associated with subdifferential operators, {C}auchy problems, J.\
  Differential Equations 46 (1982) 268--299.

\bibitem{Ota93}
M.~\^Otani, Nonlinear evolution equations with time-dependent constraints,
  Adv.\ Math.\ Sci.\ Appl. 3 (1993/94) 383--399.

\bibitem{OY78}
M.~\^Otani, Y.~Yamada, On the {N}avier--{S}tokes equations in noncylindrical
  domains: an approach by the subdifferential operator theory, J.\ Fac.\ Sci.\
  Univ.\ Tokyo Sect.\ IA Math. 25 (1978) 185--204.

\bibitem{Sim87}
J.~Simon, Compact sets in the spaces ${L}^p(0,{T}; {B})$, Ann.\ Mat.\ Pura.\
  Appl.~(4) 146 (1987) 65--96.

\bibitem{Tem01}
R.~Temam, Navier--Stokes Equations, AMS Chelsea Publishing, 2001.

\bibitem{Tsu2024}
S.~Tsubouchi, Continuous differentiability of a weak solution to very singular
  elliptic equations involving anisotropic diffusivity, Adv.\ Calc.\ Var. 17
  (2024) 881--939.

\bibitem{Wlo1987}
J.~Wloka, Partial Differential Equations, Cambridge University Press, 1987.

\end{thebibliography}

\end{document}